\documentclass[12pt]{amsart}

\usepackage{graphicx}

\usepackage{amssymb}
\usepackage{enumitem}
\usepackage{marginnote}
\usepackage{comment}
\usepackage{tikz-cd}
\input xy\huge
\xyoption{all}
\usepackage{mathtools}
\usepackage{csquotes}
\usepackage{mathrsfs}
\usepackage{xcolor}
\definecolor{will}{RGB}{8,166,241}
\definecolor{will2}{RGB}{241,83,8}
\definecolor{willcorr}{HTML}{552afa}
\definecolor{willcorr2}{HTML}{006400}
\usepackage{hyperref}
\hypersetup{
  colorlinks=true,
  linkcolor=will2,
  urlcolor=will,
  citecolor=will,
  }
\usepackage{tcolorbox}

\newcommand{\R}{\mathbb{R}}

\newcommand{\T}{\mathbb{T}}

\newcommand{\Z}{\mathbb{Z}}

\newcommand{\K}{\mathcal{K}}

\DeclareMathOperator{\Log}{Log}

\DeclareMathOperator{\val}{val}

\setlist[enumerate,1]{label={(\alph*)}}
\setlist[enumerate,2]{label={(\roman*)}}

\newcommand{\cK}{\mathcal{K}}
\newcommand{\cH}{\mathcal{H}}
\newcommand{\cO}{\mathcal{O}}
\newtheorem{thm}{Theorem}[section]

\newtheorem{pr}[thm]{Proposition}
\newtheorem{lem}[thm]{Lemma}
\newtheorem{cor}[thm]{Corollary}

\theoremstyle{definition}
\newtheorem{df}[thm]{Definition}

\theoremstyle{remark}
\newtheorem{rem}[thm]{Remark}
\newtheorem{ex}[thm]{Example}

\title[Symplectic cohomology of magnetic cotangent bundles]{The symplectic cohomology of magnetic cotangent bundles}
\author{Yoel Groman and Will J. Merry}
\date{\today}

\address{ Yoel Groman\\
Hebrew University of Jerusalem\\
Mathematics Department}
\email{ygroman@gmail.com}

\address{Will J.~Merry\\
 Department of Mathematics\\
 ETH Z\"urich}
\email{merry@math.ethz.ch}

\begin{document}
\maketitle{}

\begin{abstract}
We construct a family version of symplectic Floer cohomology for magnetic cotangent bundles, without any restrictions on the magnetic form, using the dissipative method for compactness introduced in \cite{Groman2015}. As an application, we deduce that if $N$ is a closed orientable manifold and $ \sigma$ is a magnetic form that is not weakly exact, then the $ \pi_1$-sensitive Hofer-Zehnder capacity of any compact set in the magnetic cotangent bundle determined by $ \sigma$ is finite.
\end{abstract}

\section{Introduction}
The study of geodesic flows on Riemannian manifolds is one of the oldest and richest fields in conservative dynamics. A natural generalisation is the notion of a \textit{magnetic geodesic flow}, which models the motion of a charged particle in a magnetic field. The symplectic framework for studying this problem involves modifying the canonical symplectic structure carried by a cotangent bundle, rather than modifying the Hamiltonian encoding the dynamics. Namely, let $N$ be a closed orientable manifold with cotangent bundle $ \pi \colon T^*N \to N$ and let $ \lambda $ denote the Liouville 1-form on $T^*N$. Then $\omega = -d \lambda$ is a symplectic form on $T^*N$. Now take a closed 2-form $ \sigma$ on $N$ (representing the magnetic field) and build from $ \sigma$ a new symplectic form
$$
\omega_{ \sigma} := \omega + \pi^* \sigma.
$$
We call $ \omega_{ \sigma}$ a \textit{magnetic symplectic form}, and we refer to the symplectic manifold $(T^*N, \omega_{ \sigma})$ as a magnetic cotangent bundle (the name ``twisted cotangent bundle'' is also often used).

The aim of this paper is to study how the symplectic topology of $(T^*M, \omega_{\sigma})$ is influenced by the choice of $\sigma$, via the lens of symplectic cohomology. Our main result is:
\begin{thm}
\label{thm:main}
Let $N$ be a closed orientable manifold and let $ \sigma \in \Omega^2(N)$ denote a closed 2-form. The symplectic cohomology $SH^*(T^*N \! \colon \! \omega_{ \sigma})$ is well-defined. Moreover there is an isomorphism of rings
\begin{equation}
  \label{eq-thm-main-iso}
  SH^*(T^*N \! \colon \! \omega_{ \sigma}) \cong H_{n-*}(\mathcal{L}N )_{ \tau(\sigma) \otimes \tau(w_2(N))}.
\end{equation}
\end{thm}
Symplectic cohomology is defined as a direct limit of Floer cohomologies of Hamiltonians with linear growth at infinity, and $H_*(\mathcal{L}N )_{ \tau(\sigma)}$ denotes the homology of the free loop space $ \mathcal{L}N$ with twisted coefficients\footnote{Throughout this paper we work with coefficients in an underlying field $\mathbb{K}$. If $\mathbb{K}$ has characteristic $2$ then the term involving the trangression of the second Stiefel-Whitney class of $N$ in \eqref{eq-thm-main-iso} can be dropped.}. See \S\ref{SecCoefficients} for details.

We say that a magnetic form $ \sigma$ is \textit{weakly exact} if the lift of $  \sigma$ to the universal cover of $N$ is exact. In this paper we will focus on the case when $ \sigma$ is \textit{not} weakly exact. This is due to the following result:
\begin{thm}[Albers, Frauenfelder and Oancea \cite{AlbersFrauenfelderOancea2017}]
\label{thm:afo}
Let $ \sigma \in \Omega^2(N)$. If $ \sigma$ is not weakly exact then the twisted loop space homology vanishes:
$$
H_{*}( \mathcal{L}N)_{ \tau( \sigma)} = 0.
$$
\end{thm}
Theorem \ref{thm:afo} is not explicitly stated in \cite{AlbersFrauenfelderOancea2017}. We outline its proof as Theorem \ref{thm:afo2} below. Thus by Theorem \ref{thm:main}, whenever $ \sigma$ is not weakly exact, the symplectic cohomology vanishes. From this we obtain our second main result.
\begin{thm}
\label{thm:hz-finite}
Let $N$ be a closed orientable manifold and let $ \sigma \in \Omega^2(N)$ denote a closed 2-form that is not weakly exact. Then the $ \pi_1$-sensitive Hofer-Zehnder capacity of any compact set $K \subset T^*N$ is finite.
\end{thm}

\begin{rem}
The assumption that $N$ is orientable in Theorems \ref{thm:main} and \ref{thm:hz-finite} can likely be dropped by making use of Abouzaid's approach \cite{Abouzaid2015} to the symplectic cohomology of cotangent bundles. In this case the right-hand side of \eqref{eq-thm-main-iso} would need further twisting by the pull-back of the orientation local system to the loop space. Nevertheless, we will not pursue this topic in the present paper.
\end{rem}

Finiteness of the $\pi_1$-sensitive Hofer-Zehnder capacity of $K$ implies almost everywhere existence of periodic orbits close to the regular energy level $ \partial K$ \cite[Chapter 4, Theorem 4]{HoferZehnder1994}. In particular, this includes the case of regular energy level sets of Tonelli Hamiltonians on $T^*N$, which gives a new proof of the following result of Asselle and Benedetti.
\begin{cor}[Asselle and Benedetti \cite{AsselleBenedetti2017}]
Let $N$ be a closed orientable\footnote{In \cite{AsselleBenedetti2017} Asselle and Benedetti do not need the orientability hypothesis.} manifold and let $ \sigma \in \Omega^2(N)$ denote a closed 2-form that is not weakly exact. Let $H$ denote a Tonelli Hamiltonian. Then for almost every $k > \operatorname{min}(H)$, the energy level $H^{-1}(k)$ carries a contractible periodic orbit.
\end{cor}

Let us briefly discuss the proof of Theorem \ref{thm:main}. We show that the symplectic cohomology of $(T^*N, \omega_{\sigma})$ is isomorphic as a ring to the symplectic cohomology of $(T^*N, \omega)$ with twisted coefficients, as defined by Ritter in \cite{Ritter2009}, who also showed the latter is isomorphic to twisted loop space homology.

Theorem \ref{thm:main} is a corollary of Theorem \ref{tmMainIso} below which constructs a symplectic cohomology sheaf over $H^2(N)$, with stalk over $[\sigma]\in H^2(N)$ given by $SH^*(T^*N \! \colon \! \omega_\sigma)$. This can be viewed as a closed string instance of family Floer cohomology, and the argument is reminiscent to that of \cite{Abouzaid2017}. From a technical point of view, Theorem \ref{tmMainIso} is the main contribution of this paper.

The question of how Floer-theoretic invariants behave under deformation of the symplectic form is highly non-trivial, even in the case where the perturbation is compactly supported (see for instance Ritter \cite{Ritter2010a} and Zhang \cite{Zhang2016}). While the applications in this paper are focused on magnetic cotangent bundles, the proof of Theorem \ref{tmMainIso} only uses the fact that the cotangent bundle is conical at infinity and that the twisting form is scaling invariant. Thus the theorem holds more generally for symplectic manifolds with an end modeled on the positive end of the symplectization of a contact manifold and with the symplectic form deformed by a $2$-form which is scaling invariant near infinity. Another major setting where our work should be applicable is hyperk\"ahler manifolds, moving beyond the ALE case covered by Ritter in \cite{Ritter2010a}.
Our argument works only for a strictly non-quantitative version of symplectic cohomology. Nevertheless, Theorem \ref{thm:hz-finite} shows it has implications for quantitative symplectic topology.
Let us comment on the relation of the construction of symplectic cohomology in Theorem \ref{thm:main} to the notion of ``universal'' symplectic cohomology in \cite{Groman2015}. Therein a universal directed system is constructed for any geometrically bounded symplectic manifold $(M,\omega)$, associating a Floer cohomology with quite general Hamiltonians which are proper and bounded from below. Furthermore, to each monoid $\mathcal{H}\subset C^0(M)$, there is associated a ring $SH^*(M;\mathcal{H}  \colon \! \omega)$. The monoid which enters in the definition of symplectic cohomology in Theorem \ref{thm:main} is that consisting of Hamiltonians which are $O(|p|)$ as $|p|\to\infty$.

This choice of $\mathcal{H}$ is crucial. In fact, if one considers Hamiltonians which are $O(|p|^2)$ as $|p|\to\infty$, a counter-example can be found by considering the example of $N=\T^2$ and $\sigma$ the area form. It is shown in \cite{FrauenfelderMerryPaternain2015} that considering $H(p,q)=a|p|^2$ the Floer cohomology $HF^*(H)$ can be made to vanish in any given degree by taking $a$ large enough.  In particular, for $\mathcal{H}=\{nH\}$ we have $SH^*(T^*N; \mathcal{H}  \colon \! \omega_\sigma)=0$. On the other hand we have in this case that $
H_{n-*}( \mathcal{L}N )_{ \tau( \sigma)}\neq 0$.

\begin{rem}
In this paper we focus only on applications to non-weakly exact magnetic forms. In a sequel to the present paper we discuss applications of Theorem \ref{thm:main} in the case where $ \sigma$ \textit{is} weakly exact. In this case there is a certain ``critical'' energy level $c( \sigma, g)$, called the  \textit{Ma\~n\'e critical value}. The dynamical behaviour of the magnetic geodesic flow varies dramatically depending whether one is above or below the critical value. Nevertheless, under a topological assumption on the manifold, the $ \pi_1$-sensitive Hofer-Zehnder capacity is always finite, and in some cases, admits a quantitative estimate. We will explore this further in \cite{GromanMerry2022}.
\end{rem}

\textit{Remark:} During the course of the preparation of this article Benedetti and Ritter released \cite{BenedettiRitter2018}, which among other things proves a version of Theorem \ref{thm:main} for surfaces. \\

We conclude this Introduction with two further applications of Theorem \ref{thm:main} that go in slightly different directions. Firstly, we also obtain existence results for \textit{non-contractible} periodic orbits. The next result is proved as Theorem \ref{thm:non-contract-proof} in Section \ref{sec:applications}.
\begin{pr}
\label{pr:non-contractible}
Let $(N,g)$ be a closed orientable Riemannian manifold and let $ \sigma$ denote a closed 2-form. Let $a \in [S^1, N]$ be a non-trivial free homotopy class. Let $A \subset (0 , \infty)$ be an open set such that there are no periodic magnetic geodesics $ q$ with length $\ell_g(q)\in A$. Then if $ \mu$ denotes the Lebesgue measure, one has
$$
\lim_{ n \to \infty} \frac{1}{n} \mu( A \cap [0,n]) = 0.
$$
\end{pr}

The main idea is that by Theorem \ref{tmMainIso}  the  symplectic cohomology for the non-magnetic case can be obtained from a chain complex generated by magnetic geodesics via an appropriate algebraic tweak of the Floer complex.

\begin{rem}
Taking $\sigma = 0$ in Proposition \ref{pr:non-contractible} shows that no closed orientable Riemannian manifold has a ``longest'' closed geodesic in a non-trivial free homotopy class\footnote{We thank Gleb Smirnov for bringing this question to our attention.}. Although this statement is likely known to experts, we were unable to locate a proof in the literature. 
\end{rem}

\begin{rem}
For a generic metric $g$ a much stronger statement can be shown by more elementary methods. Namely, for each non-trivial free homotopy class $a$ there is a magnetic geodesic for all high enough energy levels. In the non-generic case Proposition \ref{pr:non-contractible} permits the existence of an infinite sequence of ``gaps'', i.e, open intervals of levels for which there are no periodic orbits in a given class. We do not know whether this actually occurs.  It is also natural to ask whether Proposition \ref{pr:non-contractible} can be upgraded to a statement concerning Biran-Polterovich-Salamon capacities. For this connection we point the reader to \cite{Gong2018} for related results.
\end{rem}
Secondly, one can also use vanishing of symplectic cohomology to obtain obstructions to Lagrangian embeddings in magnetic cotangent bundles. Here is an application of this principle, which is proved as Theorem \ref{thm:lagapp} in Section \ref{sec:applications}.
\begin{pr}
\label{pr:lagrangian}
Let $N$ be a closed orientable manifold of finite type and let $ \sigma \in \Omega^2(N)$ denote a closed 2-form that is not weakly exact. Suppose $ i \colon L \hookrightarrow T^*N$ is a closed Lagrangian embedding. Then $i_* \colon H_*(L ; \mathbb{Z}) \to H_*(T^*N ; \mathbb{Z})$ is not injective.
\end{pr}
The proof uses the fact that if $L$ is $H_*$-injective the machinery of \cite{FukayaOhOhtaOno2009} allows for the definition and non-vanishing of $HF^*(L,L)$ as well as the existence of a unital ring map $SH^*(T^*N \! \colon \! \omega_{ \sigma}) \to HF^*(L,L \! \colon \! \omega_{ \sigma})$ (the closed-open map). There are several variants of Proposition \ref{pr:lagrangian}. For example, instead of $H_*$-injectivity, it suffices to assume that $L$ is simply connected.  A closely related statement which does not rely on the machinery of \cite{FukayaOhOhtaOno2009} is Theorem \ref{thm:LagBdDisc} below stipulating that \textit{any} closed Lagrangian must bound holomorphic discs of a priori bounded energy.

\begin{rem}
Finally, let us briefly consider non-closed Lagrangian submanifolds. This result is technically more challenging, as it requires the construction of the \textit{twisted wrapped Floer homology}, and the full details will appear in \cite{GromanMerry2022}.

If $S \subset N$ is a closed submanifold then the \textit{conormal bundle} $C^*S \subset T^*N$ consists of pairs $(q,p)$ such that $p|_{T_qS} = 0$. Thus if $S$ is a point $q$ then $C^*S = T_q^*N$, and if $S = N$ then $C^*N$ is the zero section. One checks easily that $C^*S$ is a Lagrangian submanifold of $(T^*N, \omega_{ \sigma})$ if and only if $ \sigma|_S = 0$; moreover $H_*(S ; \mathbb{Z}) \to H_*(N ; \mathbb{Z})$ is injective if and only if $H_*(C^*S; \mathbb{Z}) \to H_*( T^*N; \mathbb{Z})$ is injective. In this case the vanishing of symplectic cohomology gives almost everywhere existence (with respect to energy) of magnetic chords that start and end in $C^*S$, whenever $ \sigma|_S =0$ and $S$ is $H_*$-injective.
\end{rem}

\textit{Outline:} In Section \ref{sec:prelim} we present various preliminaries on the geometry of magnetic cotangent bundles. In Sections \ref{sec:disspative}-\ref{sec:estimate_proofs} we construct the symplectic cohomology groups $SH^*(T^*N \! \colon \! \omega_{ \sigma})$ and prove Theorem \ref{thm:main}. Finally, Proposition \ref{thm:hz-finite}, and Proposition \ref{pr:non-contractible} and Proposition \ref{pr:lagrangian} and  are proved in Section \ref{sec:applications}.

\textit{Acknowledgements:} We thank Mohammed Abouzaid, Gabriel P. Paternain, Jake. P. Solomon and Gleb Smirnov for helpful discussions. The first author was supported during various stages of the work by the ISF Grant 1747/13, Swiss National Science Foundation (grant \#200021$\_$156000) and by the Simons Foundation/SFARI ($\#$385571,M.A.). The second author was supported by the Swiss National Science Foundation (grant \#182564). 
\section{Preliminaries}
\label{sec:prelim}
\subsection{The geometry of magnetic cotangent bundles}

We begin by collecting some elementary results on the geometry of magnetic cotangent bundles.

Let $N$ denote a closed orientable manifold and $\pi \colon T^*N \to N$ the cotangent bundle. We write a point $x \in T^*N$ as a pair $x = (q,p)$, so that $q \in N$ and $p \in T_q^*N$. Denote by $\lambda \in \Omega^1(T^*N)$ the \textit{Liouville 1-form}, defined by
\[
  \lambda_{(q,p)}(\xi) := p( d \pi(q,p)[ \xi]), \qquad \xi \in T_{(q,p)}T^*N.
\]
We will denote as before $\omega := - d \lambda$, which we refer to as the canonical symplectic form on $T^*N$.
Fix a Riemannian metric $g = \left\langle \cdot,\cdot \right\rangle$ on $N$, and denote by $\nabla$ the Levi-Civita connection on $g$. The metric  gives an isomorphism $TN \cong T^*N$, and it what follows we will wherever convenient abuse notation and regard this  isomorphism as an identification.
The metric also determines a vector bundle isomorphism
\begin{equation}
\label{eq:splitting}
  TT^*N \cong TN\oplus T^*N
\end{equation}
Strictly speaking, the right-hand side is the vector bundle $\pi^*TN \oplus \pi^*T^*N$ over $T^*N$, whose fibre over $(q,p)$ is $T_q N \oplus T_q^*N$. This isomorphism is given by
\[
  \xi  \mapsto (d \pi (\xi), \kappa(\xi)),
\]
where $\kappa :TT^*N \to T^*N$ is the connection map of the Levi-Civita connection. Given a curve $x(t)$ in $T^*N$, writing $x(t) = (q(t),p(t))$, the connection map $\kappa$ is given by
\[
  \kappa(\dot x(t)) = \nabla_t p(t).
\]
and thus the isomorphism \eqref{eq:splitting} is given by
\[
  \dot x (t) \cong ( \dot q (t) , \nabla_t p (t)).
\]
We will use the notation $\xi \cong (\xi^q, \xi^p)$ to indicate that a tangent vector $\xi \in TT^*N$ is identified with the pair $(\xi^q, \xi^p) \in TN \oplus T^*N$ under this splitting. A vector $\xi$ with $\xi^q = 0$ is called \textit{vertical} and a vector $\xi$ with $\xi^p = 0$ is called \textit{horizontal}. This splitting allows us to define a natural Riemannian metric $\widetilde{g}$ on $T^*N$, called the \textit{Sasaki metric}, by setting
\[
\widetilde{g}(\xi, \zeta) := \left\langle \xi^q, \zeta^q \right\rangle	 + \left\langle \xi^p, \zeta^p \right\rangle.
\]
The symplectic form $\omega$ can be written in terms of this splitting as
\begin{equation}
\label{eq:omega_in_splitting}
  \omega(\xi, \zeta) = \left\langle \xi^q, \zeta^p \right\rangle - \left\langle \zeta^q, \xi^p \right\rangle.
\end{equation}

Let $E_g \colon T^*N \to \mathbb{R}$ denote the autonomous energy Hamiltonian given by the Riemannian metric:
\begin{equation}
\label{eq:G-Ham}
  E_g(q,p) := \frac{1}{2}|p|^2.
\end{equation}
Let $\sigma$ be a closed $2$-form on $N$, and write as before, $$
\omega_{ \sigma} := \omega + \pi^* \sigma.
$$
We denote by $X_{\sigma,g}$ the vector field dual to $dE_g$ under $ \omega_{ \sigma}$:
$$
\omega_{ \sigma}(X_{ \sigma,g}, \cdot) = dE_g( \cdot).
$$
Denote by $ \varphi_{ \sigma,g}^t \colon T^*N \to T^*N$ the flow of $X_{ \sigma,g}$. We call $ \varphi_{ \sigma,g}^t$ a \textit{magnetic geodesic flow}. Note that if $ \sigma = 0$ then $ \varphi_{0,g}^t$ is the standard (co-)geodesic flow of $g$. As in the Riemannian case, by a \textit{magnetic geodesic} on $N$, we mean the projection of a flow line of $ \varphi_{ \sigma,g}^t$ to $N$. A magnetic geodesic is \textit{periodic} if it is the projection of a closed orbit of $ \varphi_{ \sigma,g}^t$. A magnetic geodesic is said to have \textit{energy} $e$ if the corresponding flow line $\gamma(t)$ of $ \varphi_{ \sigma,g}^t$ satisfies $E_g(\gamma(t)) = e$.

The \textit{Lorentz force} is the map $F = F_{\sigma, g} :TN \to TN$ defined by
\begin{equation}
\label{eq:lorentz_force}
  \sigma_q(u,v) = \left\langle F_q(u), v \right\rangle.
\end{equation}

In terms of the splitting \eqref{eq:splitting}, one has
\[
  \omega_{\sigma}(\xi, \zeta) = \left\langle \xi^q, \zeta^p \right\rangle - \left\langle \zeta^q, \xi^p \right\rangle + \left\langle F (\xi^q), \xi^p \right\rangle.
\]
Thus the vertical vectors still form a Lagrangian distribution on $TT^*N$, whereas the horizontal ones do not. The presence of a Lagrangian distribution implies that orientable magnetic cotangent bundles, like normal cotangent bundles, have vanishing first Chern class (cf \cite[Example 2.10]{Seidel2000}):
\begin{equation}
\label{eq:c1=0}
c_1(T^*N, \omega_{\sigma})=0.
\end{equation}
\subsection{Almost complex structures}\label{SeqACS}
Denote by $D_{R,g}^*N = D_R^*N \subset T^*N$ the disc bundle of radius $R$. Let $S^*N:= S^*_gN = \partial D_1^*N$ and $P^*N:=T^*N \setminus D_1^*N$. Define a diffeomorphism $\psi \colon (0,\infty)\times S^*N\to P^*N$ by $(r,x)\mapsto e^rx$. The first component of $\psi^{-1}$, namely, the function $r:=\ln|p|$, is referred to as the \emph{radial} coordinate determined by the metric $g$. Let
\[
\alpha:=\frac1{|p|}\sum p_idq_i,
\]
then, in the coordinates determined by $\psi$, $\alpha$ is translation invariant, the Liouville form is given as $e^r\alpha$ and the canonical symplectic form is
\[
-d(e^r\alpha)=-e^r(dr\wedge\alpha+d\alpha).
\]

An $\omega$-compatible almost complex structure $J$ is said to be \textit{conical at infinity} if it is invariant on $P^*N$ under translations of the radial coordinate, preserves $\ker\alpha\cap\ker dr,$ and satisfies 
\begin{equation}\label{eqMetComp}
dr \circ J = \alpha. 
\end{equation}
Denote by $g_{J,\omega}$ the metric $\omega(\cdot,J\cdot)$. Then on $P^*N$, we have
\[
g_{J,\omega}=e^r(dr^2+\alpha^2+d\alpha(J,\cdot)),
\]
where for a $1$ form $\beta$ we write $\beta^2:=\beta\otimes\beta$. Define a metric on the sphere bundle $S^*N$ by
\[
g_S:=\frac14(\alpha^2+d\alpha(J,\cdot)).
\]
Then letting $t^2:=4e^r$ we can rewrite the metric on $P^*N$ as
\begin{equation}\label{eqConicalForm}
dt^2+t^2g_S.
\end{equation}
In summary, we have proven
\begin{lem}
The metric $g_{J,\omega}$ is isometric to that of the cone over $S^*N$ equipped with the Riemannian metric $g_S$.
\end{lem}
As a particular consequence, the metric $g_{J,\omega}$ has bounded geometry.

We summarise various features we shall need. They are all immediate consequences of \eqref{eqConicalForm}
\begin{enumerate}
  \item
  The distance of a point on $P^*N$ to the unit sphere bundle is given by the function $t-2$. Indeed, we have \eqref{eqConicalForm} that $\|\nabla t\|=1$ so gradient curves are geodesics. Observe that the function $|p|$ is thus, roughly, the \textit{square} of the distance.

  \item
  The projection $\pi$ factors through translations in $r$.
  \item Translating a vector in the kernel of $dt$ from $r=r_0$ to $r=r_1$ has the effect of rescaling its norm by $e^{(r_1-r_0)/2}$. In particular, translation-invariant vector fields on $P^*N$ grow linearly with the distance.
  \item {\label{sCitD}}For an $m$-form $\sigma$ on $N$, the norm of $|\pi^*\sigma|$ decays roughly as $e^{-\frac{m}{2}r}$, or, as the inverse of the $m$th power of the distance from $S^*N$.
  \item
  The Hamiltonian flow with respect to $ \omega = \omega_0$ of $|p|$ is invariant under translations.
  \item
  Fix a magnetic form $\sigma$. Then there exists a family $J_s$ of almost complex structures such that $J_s$ is compatible with $\omega_{s\sigma}$ and for each $s$ we have that 
  \begin{equation}\label{eqAlmostConical}
  \|J_s-J\|_{T^*N\setminus D^*_RM}\to0
  \end{equation}
   uniformly in $C^{ \infty}$ with respect to $g_{J,\omega}$ as $R \to\infty$. To see this, first construct any smooth family $J_s$ of $\omega_{\sigma_s}$ compatible almost complex structures on the restriction $TT^*N|_{S^*N}$. For any $u>1$ let $S^*_uN$ be the sphere bundle of radius $u$. Let $\phi_u:P^*N\to P^*N$ be the Liouville scaling map given by $t\mapsto t+u$. Then $\phi_u$ maps $S^*N$ to $S^*_uN$ intertwining $\omega_{s\sigma/u}$ with $\omega_{\sigma}$. We extend $J_s$ to the outside of the unit sphere bundle by taking $J_s|_{S^*_uN}:=\phi_{u,*}J_{s/u}$, and then  extend $J_s$ to the unit disc bundle arbitrarily (but with smooth dependence on $t$). The resulting almost complex structure is as required. In particular, $g_{J_s, \omega_{s \sigma}}$ is quasi-isometric to $g_{J, \omega}$. 
   An $\omega_{s\sigma}$-compatible almost complex structure $J_s$ satisfying \eqref{eqAlmostConical} with respect to some conical almost complex structure $J$ will be referred to as being \emph{roughly conical near infinity.}
\end{enumerate}

We remark that asking an $ \omega$-compatible almost complex structure to be conical at infinity makes no requirements (other than compatibility) on the interior of $D^*_1N$. In fact as the next lemma shows, for any magnetic for $ \sigma$ on $N$, we can choose a conical $J$ such that the $L^{ \infty}$-norm of $ \pi^* \sigma$ with respect to the metric $g_{J, \omega}$ is as small as we like. This will be useful in the proof of Theorem \ref{tmMainIso} below.

\begin{lem}\label{lmRescale}
For any magnetic form $\sigma$ and any real number $\delta>0$ there exists an almost conical $\omega_{\sigma}$-comptible almost complex structure $J_{\sigma}$ such that with respect to the metric   $g_{J_{\sigma},\omega_{\sigma}}$ we have $ \| \pi^* \sigma \|_{ \infty} < \delta$.  
\end{lem}
\begin{proof}
By scaling the metric on $N$ we can assume that $\|\sigma\|_{\infty}\ll\delta$ with respect to the Riemannian metric $g$ on $N$. The projection $\pi:T^*N\to N$ is a Riemannian submersion with respect to the Sasaki metric $\widetilde{g}$. Thus, denoting by $\|\cdot\|^{\mathrm{Sasaki}}_{\infty}$ the induced norm on $TT^*N$,  we have $\|\pi^*\sigma\|^{\mathrm{Sasaki}}_{\infty}\ll\delta$. The Sasaki metric however is not conical at infinity. To fix this, denote by $\widetilde{J}$ the almost complex structure mapping $\xi^q$ to $\xi^p$ and $\xi^p$ to $-\xi^q$, so that $\omega(\widetilde{J}\cdot ,\cdot ) = \widetilde{g}.$ Then the restriction of $\widetilde{J}$ to the unit sphere bundle $S^*N$ satisfies \eqref{eqMetComp}. This means that if we extend $\widetilde{J}$ from $S^*N$ to $P^*N$ by translation in $t$ we obtain a continuous conical almost complex structure $J'_0$ on $TN$ which is smooth away from $S^*N$ and such that in the unit disc bundle $\omega(\widetilde{J}\cdot ,\cdot ) = \widetilde{g}.$ By slightly modifying $J'_0$ near $S^*N$, we obtain a smooth conical almost complex structure $J_0$ for which $\|\pi^*\sigma\|^{J_0}_{\infty}\ll\delta$. 

We are still not done, since $J_0$ is compatible with $\omega$ but not with $\omega_{\sigma}$. However, using the technique of \cite[Proposition 2.5.6]{McDuffSalamon2017} we can find an  almost conical almost complex structure $J_\sigma$ which is $\omega_\sigma$ compatible, such that $g_{J_0,\omega}$ is everywhere close to $g_{J_{\sigma},\omega_{\sigma}}$. The procedure works as follows. Let $A:TT^*N\to TT^*N$ be the unique endomorphism for which 
$$
\omega_{\sigma}(\cdot,\cdot)=g_{J_0,\omega}(A\cdot,\cdot).
$$ 
Since $\omega(\cdot,\cdot)=g_{J_0,\omega}(-J_0\cdot,\cdot)$ and since $\|\pi^*\sigma\|^{g_{J_0,\omega}}_{\infty}\ll\delta$ we have $$A=-(J_0+\epsilon)$$ for some $\epsilon:TT^*N\to TT^*N$ satisfying $\|\epsilon\|^{g_{J_0,\omega}}_{\infty}\ll\delta$. The automorphism $-A^2$ is positive definite, so there is a unique positive definite automorphism $Q:TN\to TN$ such that $Q^2=-A^2$. Take $J_{\sigma}:=A^{-1}Q$. Then it is shown in the proof of \cite[Proposition 2.5.6]{McDuffSalamon2017} that $J_{\sigma}$ is $\omega_{\sigma}$-compatible. 

By construction, it is clear that $$\|J_{\sigma}-J_0\|^{g_{J_0,\omega}}_{T^*N\setminus D^*_RM}=O(\|\pi^*\sigma\|_{T^*N\setminus D^*_RM}).$$ From this and \ref{sCitD} it follows that $J_{\sigma}$ is almost conical. We have that $g_{J_\sigma,\omega_\sigma}$ is everywhere $(1+\delta)$-equivalent to $g_{J,\omega}$ by the fact that $\omega_{\sigma}$ and $J_{\sigma}$ are everywhere $\delta$-close to $\omega$ and $J_0$ respectively. Therefore we have $ \| \pi^* \sigma \|_{ \infty} < \delta$ with respect to $g_{J_\sigma,\omega_\sigma}$ as well.
\end{proof}

\subsection{Radial Hamiltonians}
\label{sec:radial}
Given a smooth symmetric function $h \colon  \R \to \R$, we define a Hamiltonian $H= H_{h,g} \colon  T^*N \to \R$ by
\begin{equation}
\label{eq:radial_hamiltonian}
  H(q,p) = h( | p |).
\end{equation}
Fix $\xi \in T_{(q,p)}T^*N$ and write $\xi = \dot x (0)$ for a curve $x(t) = (q(t),p(t))$. Then
\begin{align*}
dH_{(q,p)}(\xi) & = \frac{\partial }{\partial t}\Big|_{t=0} h( |p(t)|) \\
& = h'(|p|) \frac{\partial }{\partial t}\Big|_{t=0} \sqrt{ \left\langle p(t), p(t) \right\rangle_{q(t)}} \\
& = \frac{ h'(|p|)}{2|p|} \frac{\partial }{\partial t}\Big|_{t=0}  \left\langle p(t), p(t) \right\rangle_{q(t)} \\
& =  \frac{ h'(|p|)}{|p|} \left\langle \nabla_t p(t), p(t) \right\rangle \Big|_{t=0} \\
& = \frac{ h'(|p|)}{|p|} \left\langle \xi^p, p \right\rangle.
\end{align*}
From \eqref{eq:omega_in_splitting} and this computation we see that if we write $X_H \cong (X_H^q, X_H^p)$ then
\[
  X_H^q(q,p) = \frac{ h'(|p|)}{|p|} p, \qquad X_H^p(q,p) = 0.
\]
In the magnetic case, if  $X_{H,\sigma}$ denotes the vector field dual to $dH$ under $\omega_{\sigma}$ then writing $X_{H,\sigma} \cong ( X_{H, \sigma}^q, X_{H,\sigma}^p)$ we have
\begin{equation}
\label{eq:X_H_sigma}
X_{H,\sigma}^q(q,p)= \frac{ h'(|p|)}{|p|} p, \qquad X_{H,\sigma}^p(q,p) = \frac{ h'(|p|)}{|p|}  F_q(p).
\end{equation}
We denote by
\[
  \mathcal{P}(H , \sigma) := \left\{ x \colon S^1 \to T^*N \mid \dot x (t) = X_{H,\sigma}(x(t)) \right\}
\]
the set of 1-periodic orbits of $X_{H,\sigma}$. For the special case $h(r) = \frac{1}{2}r^2$, these are called \textit{magnetic geodesics} (since in this case $H$ agrees with the Hamiltonian $E_g$ from \eqref{eq:G-Ham}).  It follows from \eqref{eq:X_H_sigma} that a loop $x(t) = (q(t),p(t))$ belongs to $\mathcal{P}=\mathcal{P}(H,\sigma)$ if and only if
\[
  \dot q = \frac{h'(|p |)}{| p|} p, \qquad \nabla_t p =  \frac{h'(| p|)}{| p|} F_q (p).
\]
Since $F$ is anti-symmetric, one sees that if $x= (q,p) \in \mathcal{P}$ then $| p(t)|$ is constant, and hence so is $ | \dot q |$. Thus given $x= (q,p) \in \mathcal{P}$ there exists $e,r \ge 0$ such that $ | \dot q | = e$ and $ |p| = r$, with $ h'(r) = e$. Thus
\[
  x = (q,p) \in \mathcal{P}(H,\sigma) \qquad \Leftrightarrow \qquad p = \frac{r}{e} \dot q , \quad \nabla_t \dot q = \frac{e}{r}F_q (\dot q ).
 \]
Given such a $x$, let $\gamma(t) := q(t/e)$ so that  $\dot \gamma \in SN$, the unit tangent bundle. Then $\gamma$ is $e$-periodic and satisfies
\[
  \nabla_t \dot \gamma = \frac{1}{r} F_{\gamma} (\dot \gamma).
\]
\begin{lem}
\label{lem:slope_not_period_closed_geodesic}
Suppose $h (r) =a r + b$ and suppose that $a$ is not the period of a closed geodesic of $g$. Then there exists $c > 0$ such that every element $x = (q,p)$ of $\mathcal{P}(H_{h, g}, \sigma)$ satisfies $|p | \le c$.
\end{lem}
\begin{proof}
If not we obtain a sequence $r_k \to \infty$ and a sequence $\gamma_k $ of unit speed curves, all of which are $a$-periodic, which satisfy $\nabla_t \dot \gamma_k = \frac{1}{r_k}F_{\gamma_k} (\dot \gamma_k)$. Since $\gamma_k$ is contained in a level set of $H$ which is proper and $\gamma_k$ is of unit speed, this equation produces estimates on derivatives of $\gamma_k$ to all orders. Thus by Arzela-Ascoli we obtain a subsequence converging in $C^{\infty}$ to an $a$-periodic unit length curve $\gamma$ satisfying $\nabla_t \dot \gamma = 0$. \end{proof}
In practice, we will typically work with Hamiltonians which take the form \eqref{eq:radial_hamiltonian} outside of some compact set of $T^*N$ with $h$ linear outside of a compact set. Of course, for such Hamiltonians, Lemma \ref{lem:slope_not_period_closed_geodesic} is still valid. The next lemma compares $X_H$ to $X_{H, \sigma}$.
\begin{lem}\label{lmMagPertbDec}
Suppose $H$ is linear at infinity. Let $\sigma$ be a magnetic form and denote by $X_{H,\sigma}$ the Hamiltonian vector field of $H$ with respect to $\omega_{\sigma}$. Then with respect to any roughly conical almost complex structure we have
\[
\|X_H-X_{H,\sigma}\|_{T^*N\setminus D^*_RM}\to0
\]
in $C^\infty$ as $R\to\infty$.
\end{lem}
\begin{proof}
  Choosing local coordinates $\{q_i\}$ on the zero section, and considering equation \eqref{eq:X_H_sigma} we see that
  \[
  X_H-X_{H,\sigma}=X_{H,\sigma}^p(q,p) = h'(|p|)  F_q\left(\frac{p}{|p|}\right).
  \]
  Note  that the coefficients of the Lorentz force in the coordinates $\{q_i,p_i\}$  depend only on the $q_i$. Applying it to the unit vector in the direction of $p$ we get  a linear combination of the vectors $\partial_{p_i}$ with uniformly bounded coefficients. With respect to a conical almost complex structure, the vectors $\partial_{p_i}$ are decaying since they are $\omega_0$-dual to the scaling invariant fields $\partial_{q_i}$.
\end{proof}

\subsection{Novikov rings}
\label{subsection:novikov}
In this paper we consider two flavours of Novikov rings. In this section we briefly explain how they are defined. Fix once and for all a field $\mathbb{K}$.

\begin{df}
The \textbf{universal Novikov ring} $\Lambda_{\operatorname{univ}}$ is defined as the set of all formal sums 
$$
\sum_{ t \in \mathbb{R}} r_t \,T^t
$$
where $T$ is a formal parameter, $t$ is a real number and $r_t \in \mathbb{K}$. These sums are required to satisfy the \textit{Novikov finiteness condition}:  
\vspace{0.75em}

\begin{displayquote}
For all $c >0$ the set $\left\{ t \in \mathbb{R} \mid r_t \ne 0, \ t < c \right\}$ is finite.
\end{displayquote}
\end{df}

The universal Novikov ring is a field. The second flavour of Novikov ring starts with a finitely generated torsion-free abelian group $\Gamma$. This construction also appeared in recent work of Zhang \cite{Zhang2016}. Let $H^1( \Gamma ; \mathbb{R}) = \mathrm{Hom}( \Gamma ; \mathbb{R})$, and suppose $ \Phi \subset H^1( \Gamma ; \mathbb{R})$ is a rational polytope (i.e. the convex hull of finitely many points). Assume to begin with that 
\begin{equation}
\label{eq-no-common-kernel}
\bigcap_{ \phi \in \Phi} \ker \phi = \emptyset.
\end{equation}

\begin{df}
Let $\mathbb{L}$ denote (another) field. We define the \textbf{Novikov ring with multi-finiteness conditions}, written $\Lambda_{\Phi}(\mathbb{L})$, to be the set of all formal sums
$$
\sum_{ \gamma \in \Gamma} s_{ \gamma}\, z^{ \gamma}
$$
where $z$ is another formal variable, and this time the coefficients $s_\gamma \in \mathbb{L}$ satisfy
\vspace{0.75em}
\begin{displayquote}
For all  $c >0$  and $\phi \in \Phi$  the set $\left\{ \gamma \in \Gamma \mid s_{ \gamma} \ne 0, \ \phi(\gamma) < c \right\}$ is finite.
\end{displayquote}
\end{df}
Thus by definition
$$
\Lambda_{ \Phi}(\mathbb{L}) = \bigcap_{ \phi \in \Phi} \Lambda_{ \phi}(\mathbb{L}),
$$
where by a slight abuse of notation we write $\Lambda_{ \phi}(\mathbb{L})$ for $\Lambda_{ \{ \phi \}}(\mathbb{L})$.
It follows from an argument analogous to \cite[Lemma 7.4]{Zhang2016} that if $ \Phi$ has vertices $ \phi_1, \dots , \phi_m$ then
\begin{equation}\label{eq_Zhang_vertices}
\Lambda_{ \Phi}(\mathbb{L})= \bigcap_{i=1}^m \Lambda_{ \phi_i}(\mathbb{L}).
\end{equation}
This implies that $ \Lambda_{ \Phi}(\mathbb{L})$ is  an integral domain. If \eqref{eq-no-common-kernel} is not satisfied, we simply replace $\Gamma$ by $\bar{\Gamma} = \Gamma/K$, where $K$ is the common kernel, and consider the induced polytope $\bar{\Phi}$ in $H^1(\bar{\Gamma};\mathbb{R})$. In order to keep the notation under control, we continue to write $\Lambda_{\Phi}(\mathbb{L})$ instead of $\Lambda_{\bar{\Phi}}(\mathbb{L})$. 

In fact, in this paper, we will always take $\mathbb{L} = \Lambda_{\operatorname{univ}}$. Thus we abbreviate
\begin{equation}
\label{eq-final-nov}
\Lambda_\Phi \coloneqq \Lambda_\Phi( \Lambda_{\operatorname{univ}}).
\end{equation}

\subsection{Loop spaces and transgressions}
\label{sec-lst}

Let
 $$ \mathcal{L}N = \bigoplus_{ a \in [S^1, M]} \mathcal{L}^aN$$
 denote the free loop space of $N$, and let $\mathcal{L}T^*N$ denote the free loop space of $ T^*N$. Fix for each $a\in[S^1,N]=[S^1,T^*N]$ a reference loop $\gamma_a=(q_a,p_a)\in \mathcal{L}^aT^*N$. We require that $\gamma_{0}$ is a constant loop and that $\gamma_{-a}(t)=\gamma_a(-t)$.
 The loop $\gamma_a$ (resp. $q_a$) will always be taken as the basepoint when referring to $ \pi_1(\mathcal{L}^aT^*N)$ (resp. $ \pi_1(\mathcal{L}^aN)$).

Define the \textit{transgression map}
$$
\tau \colon H^2(N; \mathbb{R}) \to H^1( \mathcal{L}N; \mathbb{R})
$$
to be the composition
$$
H^2(N ; \mathbb{R}) \to H^1(\mathcal{L}N \times S^1 ; \mathbb{R}) \to H^1( \mathcal{L}N ; \mathbb{R}),
$$
where the first map is induced by the evaluation $ \mathcal{L}N \times S^1 \to N$, $(q,t) \mapsto q(t)$, and the second projects onto the K\"unneth summand.

Explicitly, if $ k \in H^2(N ; \mathbb{R})$ and $ \sigma$ is a closed 2-form with $ [ \sigma ]  = k$, then the transgression class $\tau(k)$ is represented by the singular 1-cochain
$$
\tau(k) \in C^1_{\operatorname{sing}}(\mathcal{L}N; \mathbb{R}), \qquad u \mapsto \int u^*\sigma,
$$
for $u$ a smooth\footnote{For an arbitrary path $u$, first homotope $u$ relative to its ends to make it smooth. The resulting integral is independent of this homotopy since $\sigma$ is closed.} path in $\mathcal{L}N$. One can heuristically interpret $\tau(k)$ as being represented by the ``one-form'' $\tau(\sigma) \in \Omega^1(\mathcal{L}N)$:
\begin{equation}
  \label{eq:transgression}
   \tau( \sigma)_q(v) := \int_{S^1}\sigma( \dot q, v), \qquad  q \in \mathcal{L}N,\  v \in C^{ \infty}( q^*TN).
\end{equation}
Let $i_a \colon \mathcal{L}^aN \to \mathcal{L}N$ denote the inclusion. A two-form $ \sigma$ is weakly exact (i.e. it lifts to an exact form on the universal cover of $N$) if and only if $  i_0^*\tau [ \sigma] = 0$, i.e. if the 1-form $\tau( \sigma)$ is exact when restricted to the space of contractible loops.

\subsection{Covers of the loop space}
\label{sec:covers}
As before, let $\sigma$ denote a closed two-form on $N$. We define a map
$$
I_{\sigma} \colon \pi_1( \mathcal{L} N) \to \mathbb{R}
$$
by
$$
I_{\sigma}[w] := \int_{\mathbb{T}^2}w^*\sigma
$$
for $w \colon S^1 \to \mathcal{L} N$ a smooth representative of  $[w] \in \pi_1( \mathcal{L} N)$, thought of as a map $w \colon \mathbb{T}^2 \to N$.

Let $\mathcal{L}_{ \sigma} N  \to \mathcal{L} N$ denote the covering space of $\mathcal{L} N$ with deck transformation group $\pi_1( \mathcal{L} N) \big/ \ker I_{ \sigma}$. An element of $\mathcal{L}_{ \sigma} N$ is an equivalence class $[q,w]$ of a pair $(q,w)$, where $q \in \mathcal{L} N$ and $w \colon [0,1] \to \mathcal{L}N$ is a path from $q$ to $q_a$ (where $ q$ belongs to the free homotopy class $a$ and $q_a$ is our earlier fixed reference loop) and the equivalence relation is given by
$$
(q,w_{0})\sim(q,w_{1}) \qquad \Leftrightarrow \qquad [w_{0} * w_{1}^{-1}]\in\ker(I_{ \sigma}),
$$
By definition of the cover $\mathcal{L}_\sigma N$, the
``1-form'' $\tau(\sigma)$ pulls back to an exact form in $\mathcal{L}_\sigma N$.

We remark that the cover $ \mathcal{L}_{ \sigma} N$ (and all the related objects) only depend on the half-line in $H^2(  N ; \mathbb{R})$ determined by $ [\sigma]$.

There is an analogous cover $ \mathcal{L}^a_{ \sigma}T^*N$ of the free loop space of the cotangent bundle of $N$, where we use the reference loops $\gamma_a$. This makes the following diagram commute:
$$
\begin{tikzcd}
\mathcal{L}_{ \sigma}T^*N \ar[r] \ar[d] & \mathcal{L}T^*N \ar[d] \\
\mathcal{L}_{ \sigma}N \ar[r] & \mathcal{L}N
\end{tikzcd}
$$
The reader may supply the details.

If we define $\mathcal{A}_{\sigma} \colon \mathcal{L}_{\sigma} N \rightarrow \mathbb{R}$
by
$$
\mathcal{A}_{\sigma}([q,w]) \coloneqq \int_{[0,1]\times S^{1}}w^{*}\sigma,
$$
then the pullback of $ \tau( \sigma)$ to $\mathcal{L}_{\sigma} N$ is formally equal to $d A_{ \sigma}$.

Moreover since $ \omega_{ \sigma} = -d \lambda + \pi^* \sigma$ belongs to the same cohomology class as $ \pi^* \sigma$, the cover $ \mathcal{L}_{ \sigma}T^*N$ of $ \mathcal{L}T^*N$ can also be thought as the cover corresponding to the kernel of $I_{ \omega_{ \sigma}}$.

\subsection{Homology with local coefficients}\label{SecCoefficients}

Next, we use $ \sigma$ to build a system of local coefficients on $ \mathcal{L}N$. Let us briefly recall the relevant definitions. We recommend \cite[Chapter 6]{Whitehead1978} for more information.

Suppose $X$ is a path connected topological space admitting a universal cover, and let $R$ be a commutative ring. A \textit{local system of $R$-modules} on $X$ is a functor $ \mathscr{C}$ from the fundamental groupoid of $X$ to the category of $R$-modules. Thus to each point $x \in X$, $ \mathscr{C}$ associates an $R$-module $ \mathscr{C}(x)$, and if $u$ is a path in $X$ from $x$ to $y$ then $ \mathscr{C}$ associates to $u$ an isomorphism
$$
\mathscr{C}(u) \colon \mathscr{C}(x) \to \mathscr{C}(y),
$$
such that $ \mathscr{C}(u)$ only depends on the path homotopy class $[u]$, and such that $ \mathscr{C}(u*v) = \mathscr{C}(v)\circ \mathscr{C}(u)$ (here $*$ denotes concatenation of paths), and finally such that $ \mathscr{C}( \mathrm{ct})$ is the identity map for any constant path $ \mathrm{ct}$. Rather than give the (somewhat lengthy) definition of how the \textit{homology of $X$ with local coefficients in $ \mathscr{C}$}, written $H_{ *}(X ; \mathscr{C})$, is defined, we simply remark that Eilenberg's Theorem (see for example \cite[Theorem VI.3.4]{Whitehead1978}) allows us to identify they groups as the equivariant homology groups of $ \widetilde{X}$ with respect to $ \mathscr{C}$:
$$
H_{ *}(X ; \mathscr{C}) = H_{ *}\left(  C_{ *}^{ \mathrm{sing}}( \widetilde{X}) \otimes_{ \mathbb{Z}[ \pi_1(X)]} R\right).
$$
We will prove a Floer-theoretic version of Eilenberg's Theorem in Proposition \ref{prEilenberg} below.
Going back to the previous setting, any class $ b \in H^1 ( \mathcal{L} N ; \mathbb{R})$ defines a local system of $ \Lambda_{ \mathrm{univ}}$-modules on $ \mathcal{L} N$. This works as follows: Pick a singular 1-cocycle representing the given class $b$. Given a path $u$ in $ \mathcal{L} N$ from say, $q_0(t)$ to $q_1(t)$, the desired isomorphism is
$$
\mathscr{C}_{ \beta}(u) \colon \Lambda_{ \mathrm{univ}}(q_0) \to \Lambda_{ \mathrm{univ}}(q_1), \qquad z \mapsto e^{-\int \beta[u]} z.
$$

This definition depends on a choice of closed cochain $\beta$ used to represent $b$. However different choices give rise to isomorphic local systems (which thus have isomorphic homology).

As with the previous subsection, we are solely interested in the case where $b = \tau [ \sigma]$ is the transgression of a closed 2-form $ \sigma \in \Omega^2(N)$ (and then we take $ \beta = \tau( \sigma)$).

The following result is not explicitly stated in \cite{AlbersFrauenfelderOancea2017}. It extends a previous result of Ritter \cite[Theorem 12]{Ritter2009}.

\begin{thm}[Albers, Frauenfelder and Oancea]
\label{thm:afo2}
Let $ \sigma \in \Omega^2(N)$. If $ \sigma$ is not weakly exact then the twisted loop space homology vanishes:
$$
H_{*}( \mathcal{L}N ; \mathscr{C}_{ \tau( \sigma)}) = 0.
$$
\end{thm}
\begin{proof}
Consider the fibration
\[
\Omega^0N \xhookrightarrow{i} \mathcal{L}^0N \xrightarrow{e} N,
\]
where $ \Omega^0N$ is the component of the based loop space containing the constant loops. By \cite[Corollary 1]{AlbersFrauenfelderOancea2017} one has
\[
H_{*}(\Omega^0N ; i^*\mathscr{C}_{ \tau( \sigma)}) = 0
\]
(here we are using the fact that $ \Lambda_{\operatorname{univ}}$ is a field). The Serre Spectral Sequence then tells us that one also has
\[
H_{*}(\mathcal{L}^0N ; i^*\mathscr{C}_{ \tau( \sigma)}) = 0
\]
(cf. \cite[Proposition 5]{AlbersFrauenfelderOancea2017}). To extend this to the entire loop space $ \mathcal{L}N$, one uses the fact that the twisted homology groups $ H_{ *}( \mathcal{L}N ; \mathscr{C}_{ \tau( \sigma)})$ carries a unital ring structure \cite[Theorem 14.3]{Ritter2013}. The unit lives in $ H_{ *}( \mathcal{L}^0 N ; \mathscr{C}_{ \tau( \sigma)})$ and thus vanishing of the homology of the contractible component implies vanishing of the entire loop space homology.
\end{proof}
For the remainder of this paper, we will abbreviate the cumbersome notation $ H_*( \mathcal{L}N ; \mathscr{C}_{ \tau( \sigma)})$ by $ H_*( \mathcal{L}N)_{ \tau( \sigma)}$. This is consistent with Theorem \ref{thm:main} from the Introduction.\section{Floer theory for dissipative Hamiltonians}
\label{sec:disspative}
\subsection{The action functional and its gradient flow}
\label{sec:begin_floer}
In this section we construct a Floer theory for \textit{dissipative} Hamiltonians on non-compact symplectic manifolds, following the approach developed by the first author in \cite{Groman2015}. Although in the present paper we are only concerned with magnetic cotangent bundles, we will work in a somewhat more general setting here. To this end, let $(M,\omega)$ be a non-compact symplectic manifold satisfying $c_1(M)=0$. As before we denote by $ \mathcal{L}M$ the free loop space of $M$.

As in Section \ref{sec:covers}, the two-form $ \omega$ determines a homomorphism
$$ I_{ \omega} \colon \pi_1( \mathcal{L}M) \to \mathbb{R},$$ and hence a cover $ \mathcal{L}_{ \omega}M \to \mathcal{L}M$, whose elements are pairs $( \gamma, w)$, where $w$ is a homotopy class of paths connecting $ \gamma$ to a fixed reference loop in the given free homotopy class containing $ \gamma$. 

For a smooth function $H:S^1\times M\to \R$ and for any $t\in S^1$ denote by $X_{H_t}$ the unique vector field satisfying $dH(t,\cdot)=\omega(X_{H_t},\cdot).$
Fixing base loops on the components of $ \mathcal{L}M$, define a functional $\mathcal{A}_H \colon \mathcal{L}M\to \R$ by
\begin{equation}\label{eqActionDef}
\mathcal{A}_{H}([\gamma,w]):=-\omega(w)-\int_{S^1}H(\gamma(t))dt.
\end{equation}

Denote by $\mathcal{P}(H)\subset \mathcal{L}M$ the set of $1$-periodic orbits of $X_H$. Denoting by
\[
\pi \colon \mathcal{L}_{ \omega}M \to \mathcal{L}M,
\]
the covering map, let
\[
\mathcal{P}_\omega(H):=\pi^{-1}(\mathcal{P}(H)).
\]
This is the same as the critical point set of $\mathcal{A}_H$. Each $\tilde{\gamma}\in \mathcal{P}_{ \omega}(H)$ carries an index $i_{\mathrm{RS}}(\tilde{\gamma})$ which is well defined up to a global shift for each homotopy class of loops. Namely, for each homotopy class $a$ fix  a trivialisation of $\gamma_a^*TM$. Then if $\tilde{\gamma}=(\gamma,w)$ trivialise $\gamma^*TM$ along $w$ by extending the trivialisation along $\gamma_a$ and take the Robin-Salomon index of $\gamma$ with respect to this trivialisation.

Given an $S^1$-parameterised family of almost complex structures $J_t$ on $M$, the gradient of $\mathcal{A}_{H}$ at $\gamma$ is the vector field
\[
\nabla\mathcal{A}_{H}(\gamma)(t):=J_t(\dot{\gamma}(t)-X_{H_t}(\gamma(t)))
\]
along $\gamma$. Note that the gradient field is independent of the choice of base paths.  A gradient trajectory is a path in (a covering of) the loop space whose tangent vector at each point is the negative gradient at that point. Explicitly a gradient trajectory is a map
\[
u:\R\times S^1\to M,
\]
satisfying Floer's equation
\begin{equation}\label{eqFloer}
\partial_su+J_t(\partial_tu-X_{H_t}\circ u)=0.
\end{equation}
We refer to such solutions as \textit{Floer trajectories}.

\subsection{Dissipative Floer data and a priori $C^0$ estimates}

\begin{df}\label{dfGoBound}
An $\omega$-compatible almost complex structure $J$ on $M$ is said to be \textit{uniformly isoperimetric} if the associated metric $g_J$ is complete and there are constants $\delta$ and $c$ such that any loop $\gamma$ satisfying $\ell(\gamma)<\delta$ can be filled by a disc $D$ satisfying $\mathrm{Area}(D)<c\ell(\gamma)^2.$ If $J$ is merely $\omega$-tame, we add the requirement that, for some $C>1$,
\[
\frac1C<\left|\frac{\omega(Jx,Jy)}{\omega(x,y)}\right|<C.
\]
\end{df}
\begin{rem}
According to \cite[\S4]{Sikorav94} the criterion is satisfied for $J$ compatible if $g_J$ is
\textit{geometrically bounded}, namely, if it has sectional curvature bounded from above and radius of injectivity bounded from below.
\end{rem}
\begin{rem}\label{remObs}
It is clear from the definition that if $g_{J_0}$ is uniformly isoperimetric and $g_{J_1}$ is equivalent to $g_{J_0}$ then $g_{J_1}$ is also uniformly isoperimetric.
\end{rem}
The significance of the condition comes from the following result, where for a $J$-holomorphic curve $ u \colon S \to M$ we denote by $E_J(u ; S)$ the energy $E_J(u ; S) := \int_S u^* \omega$. We omit both the $J$ and the $S$ when they are clear from the context.
\begin{lem}[\textbf{Monotonicity \cite[\S4]{Sikorav94}}]
\label{lem:monotonicity}
Suppose $g_J$ is $(\delta,c)$-isoperimetric at $p$. Then any $J$-holomorphic curve $u$ passing through $p$ and with boundary in $M\setminus B_{\delta}(p)$ satisfies
\[
E_J(u;u^{-1}(B_{\delta}(p)))\geq \frac{\delta^2}{2c}.
\]
\end{lem}
Henceforth let $J$ be $\omega$-tame and uniformly isoperimetric. Let $H:S^1\times M\to\R$ be proper and bounded from below. Assume the Hamiltonian flow of $H$ is defined for all times. This certainly holds for time independent Hamiltonians as well as for Hamiltonians satisfying $\partial_tH<aH+b$ for some constant $a,b$. Indeed, for any path $\gamma$ satisfying $\dot\gamma = X_{H_t}$ we have $\frac{d}{dt}H(\gamma(t))=\partial_tH\circ\gamma \leq aH+b$ implying an exponential estimate on $H\circ\gamma$ by Gronwall's inequality.
\begin{df}\label{dfJH}
Consider the manifold $\tilde{M}:=M\times S^1\times\R$ with the $2$-form $\omega_H:=\omega+dH\wedge dt$. The $2$-form $\omega_H$ together with the obvious projection $\pi:\tilde{M}\to\R\times S^1$ define a horizontal distribution on $\tilde{M}$ which we denote by $p\mapsto H_p$. Let $J_H$ be the almost complex structure on $\tilde{M}$ determined by the following conditions:
\begin{enumerate}
\item $J_H$ is split with respect to the splitting $T\tilde{M}=TM\oplus H$ defined by $\omega_H$.
\item $J_H|_{TM}=J$.
\item $d\pi\circ j=J\circ d\pi$.
\end{enumerate}
\end{df}
$J_H$ can alternatively be described in matrix form by
\[
J_{H}(z,x):=\left(\begin{matrix} j_{\Sigma}(z) & 0  \\ -ds\otimes X_H(z,x)-dt\otimes J(z,x)\circ X_H(z,x) & J(x) \end{matrix}\right)
\]
Define $\tilde{\omega}_H:=\omega_H+ds\wedge dt$. Observe that $J_H$ is compatible with (resp. tamed by) $\tilde{\omega}_H$ if $J$ is compatible with (resp. tamed by) $\omega$.

Our interest in the almost complex structure $J_H$ stems from the fact that there is a bijection between Floer trajectories $u$ and $J_H$-holomorphic sections $\tilde{u}$ of $\tilde{M}\to\R\times S^1$ given by taking a trajectory to its graph. This allows us to apply the Monotonicity Lemma \ref{lem:monotonicity} to Floer trajectories.

\begin{df}\label{dfDess}
We say that the pair $H,J$ is \textit{strongly dissipative} if the following two conditions are satisfied:
\begin{enumerate}
\item \label{dfDess1} The metric on $g_{J_H}$ determined on $\tilde{M}$ by $\tilde{\omega}_H:=\omega_H+ds\wedge dt$ and $J_H$, is uniformly isoperimetric.
\item \label{dfDess2} There is an $\epsilon>0$ and a compact subset $K\subset M$ such that any loop $\gamma:S^1\to M$ which intersects the complement of $K$ must satisfy
  \[
  E_{H,J}(\gamma):=\int_0^1\|\gamma'(t)-X_H\circ\gamma(t)\|^2dt>\epsilon.
  \]
\end{enumerate}
\end{df}
\begin{rem}\label{remLyapunov}
The last condition is satisfied when the flow $\varphi$ of $X_{H_t}$ is complete (in particular, if it is time independent outside a compact set). Indeed, in this case there is a constant $\lambda$ so that
$$
d(\varphi^t(x),\varphi^t(y))\leq e^{\lambda t}d(x,y),
$$
 and there are a compact set $K'$ and a constant $\epsilon'>0$ such that $d(\varphi^1_H(x),x)>\epsilon'$ for $x\in M\setminus K'$. For details see \cite[Lemma 6.17]{Groman2015}.
\end{rem}
We will see in Theorem \ref{tmLinDiss} below that these conditions are satisfied for linear Hamiltonians on the magnetic cotangent bundles. 

For a Floer trajectory $u$ let
\begin{equation}
E_{H,J}(u):=\int_{\R\times S^1}\|\partial_su\|^2=\int_{\R\times S^1}\|\partial_tu-X_H\|^2.
\end{equation}

For a measurable subset $S\subset\R\times S^1$ we denote $E_{H,J}(u;S)$ the restriction of the above integral to $S$. Given a Floer trajectory $u$, consider the corresponding $J_H$-holomorphic section $\tilde{u}$. The energy $E_{H,J}(u)$ is related to the $\tilde{\omega_H}$ energy of $\tilde{u}$ by the relation
\begin{equation}\label{eqEnergyIdArea}
E(\tilde{u};S):=\int_S\tilde{u}^*\tilde{\omega}_H=E_{H,J}(u)+ \operatorname{Area}(S).
\end{equation}
For a reference see \cite[Ch. 8]{McDuffSalamon2004}.
We have the following property established in \cite{Groman2015}.
\begin{thm}\label{tmDissDaimEst}
Fix a strongly dissipative Floer datum $H,J$. For each $E$ there is a compact set $K_E$ such that $u\subset K_{E_{H,J}(u)}$ for any Floer trajectory of $H,J$ satisfying $E_{H,J}(u) < E$.
\end{thm}
\begin{proof}
One uses the condition \ref{dfDess2} in Definition \ref{dfDess} to first produce a compact set $K$ such that $u(s,\cdot)$ intersects $K$ for any $s$. Namely, suppose $[a,b]\times S^1$ maps outside of $K$. Then we must have $b-a\leq \frac{E}{\epsilon}$. Taking $[a,b]$ to be a maximal such interval, we have that $u(a,\cdot)$ intersects $K$. Applying Cauchy-Schwartz, one has
\[
d(u(s_1,\cdot),u(s_2,\cdot))^2\leq (s_2-s_1)E(u;[s_1,s_2]\times S^1),
\]
for any pair $s_1<s_2\in\R$. Combining the two estimates gives rise to $K$. We now rely on condition \ref{dfDess1} and apply monotonicity to obtain an estimate on the diameter of $u(s,\cdot)$ for any $s$. For this consider the restriction of the graph $\tilde{u}$ of $u$ to $[s-1,s+1]\times S^1$. Considered as a $J_H$-holomorphic map, it has energy estimated by $E(u)+2\pi$. Moreover, since $\tilde{u}$ is a section, if $\tilde{u}(s,t)$ meets a point in $\tilde{M}$, the boundary $[s-1,s+1]\times S^1$ maps to the complement of the ball of radius $1$ centred at that point. Thus, monotonicity gives an a priori estimate on the number of disjoint unit balls intersected by $u(s,0)$.
\end{proof}
\begin{rem}\label{rmRobust}
By examining the proof, it becomes clear that the claim can be made \textit{domain-local} in the sense that given bounded open $A\Subset B\subset\R\times S^1$, we get a $C^0$ estimate on $u(A)$ in terms of the energy of $B$ and the geometry of $A,B$. The dependence on $A,B$ can be reduced to a dependence on $d(A,B)$ if $A$ and $B$ are sub-cylinders. The domain locality is a source of great flexibility and robustness.
\end{rem}

Let $ \Gamma$ denote the image in $H_2(N ; \mathbb{Z}) \big/ \mathrm{Tor}$ represented by maps from spheres and tori. Denote by $CF^*(H,J,\omega)$ the free $\Lambda_{univ}$ module
$$CF^*(H,J \! \colon \! \omega) \coloneqq \mathbb{K}\langle\mathcal{P}_\omega(H)\rangle\otimes_{\mathbb{K}[\Gamma]} \Lambda_{\operatorname{univ}}.$$  Here $\Lambda_{\operatorname{univ}}$ is a $\mathbb{K}[\Gamma]$-module with the action $\gamma\cdot T^{r}=T^{r+I_\sigma(\gamma)}$ and the action on $\mathcal{P}_\omega(H)$ is via deck transformations.
In this section, $\omega$ is fixed, and hence we omit it from the notation. Starting in Section \ref{sec:deformations} we will start to vary $\omega$, and hence it will return to our notation then.
Note that  $CF^*(H,J)$ is a free $\Lambda_{\operatorname{univ}}$-module.
As a consequence of the discussion in the previous paragraph, using standard constructions, we obtain the following theorem.
\begin{thm}
For $(H,J)$ dissipative, there is a well defined differential on $CF^*(H,J)$. This gives rise to the Floer homology $HF^*(H,J)$. The valuation on $CF^*(H,J)$ gives rise  to an induced valuation on $HF^*(H,J)$.
\end{thm}

\begin{tcolorbox}
Henceforth and without further notice, all Hamiltonians are implicitly assumed to be proper and bounded from below.
\end{tcolorbox}

\subsection{Zig-Zag homotopies and continuation maps}
Let $H_0\leq H_1$ be a pair of time-dependent Hamiltonians and let $J_1,J_2$ be time-dependent almost complex structures. A continuation Floer datum for $(H_i,J_i)$ is a family $(H_s,J_s)$ such that $(H_s,J_s)=(H_0,J_0)$ for $s\ll0$, $(H_1,J_1)$ for $s\gg0$, and $\partial_sH$ has infimum $-c>-\infty$. To a such a Floer datum $\mathfrak{H}$ we associate a symplectic form $\tilde{\omega}_{\mathfrak{H}}$ on $\tilde{M}:=M\times \R\times S^1$ defined as
\[
\tilde{\omega}_{\mathfrak{H}}:=\omega+dH\wedge dt+(\partial_sH+c+1)ds\wedge dt.
\]

We also have an almost complex structure $J_\mathfrak{H}$ which is compatible with $\tilde{\omega}_{\mathfrak{H}}$ and is defined as in Definition \ref{dfJH} above, only this time allowing $H,J$ to depend on $s$.
\begin{df}
Let $(H_i,J_i)$ be strongly dissipative. We call a continuation Floer datum $(H_s,J_s)$  interpolating between them \textit{dissipative} if for each $(s,t)\in\R\times S^1$ there is a neighborhood $U\subset \R\times S^1$ and an exhaustion $K_i$ of $M$ such that for some $\delta>0$ the metric determined by $\tilde{\omega}_{\mathfrak{H}}$ and $J_{\mathfrak{H}}$ is uniformly isoperimetric on the product of $U$ with a $\delta$ neighborhood of $\bigcup_{i=1}^{\infty}\partial K_i$.
\end{df}
\begin{rem}
Note that in this case we do not require that the associated metric be uniformly isoperimetric on all of $\tilde{M}$. The reason is that one cannot in general expect there to be an interpolating Floer datum satisfying such a requirement. Indeed, the space of compatible geometrically bounded almost complex structures is not known to be connected, and there is every reason to expect that it is not. More to the point, when considering a pair $H_1\leq H_2$ for which the growth rate at infinity of $H_2$ is strictly greater that that of $H_1$ it is not generally known how to construct a continuation datum whose associated Gromov metric. 
\end{rem}

Note again that by definition if a given continuation datum is dissipative, so is any continuation datum which sufficiently close to it.
\begin{thm}
The statement of Theorem \ref{tmDissDaimEst} is true upon replacing the strongly dissipative Floer datum $(H,J)$ by a dissipative continuation datum $(\mathfrak{H},J)$ interpolating between a pair of strongly dissipative Floer data.
\end{thm}
\begin{proof}
The only difference to the previous case is in the application of monotonicity. Namely, we only have monotonicity for points on $\partial K_i$. This is sufficient to obtain an a priori estimate on the number of $\partial K_i$ met by $u(s,\cdot)$. For more details see \cite[Theorem 4.11]{Groman2015} and its proof.
\end{proof}

 \begin{thm}\label{tmContinuatiation}
 Let $(H_1,J_1)$ and $(H_2,J_2)$ be dissipative Floer data such that $H_1\leq H_2$. Then there exists a natural continuation map $HF^*(H_1,J_1)\to HF^*(H_2,J_2)$ which is functorial with respect to $\leq$. Moreover, if $H_2-H_1$ is bounded, the continuation map is an isomorphism.
 \end{thm}
 \begin{proof}
 We restrict ourselves to the case where there is a fixed Riemannian metric $g$ such that $g_{J_1},g_{J_2}$ are both equivalent the metric $g$. This is the generality we will need for our application. The more general statement is proven in \cite{Groman2015}.

 Let $ (s,t) \in \R^2 \xrightarrow{f} [0,1]$ be a function satisfying $f\equiv 0$ on
\[
  \{s\leq 0\} \cup \bigcup_{n=1}^\infty[0,2/3]\times[4n,4n+1],
\]
$f\equiv 1$ on
\[
\{s\geq 1\} \cup \bigcup_{n=1}^\infty[1/3,1]\times[4n+2,4n+3],
\]
and that $ \partial_s f \ge 0$ for any $(s,t)\in\R^2$.

Fix a point $x\in M$ and let $ \eta \colon  M\to\R$ be a smoothing of the distance function to $x$. Let $J_s$ be a family of almost complex structures interpolating between $J_0$ and $J_1$ and such that $g_{J_s}$ is equivalent to $J$. Let
$$ H_s(x):=f(s, \eta(x))H_1(x)+(1-f(s,\eta(x)))H_0(x).$$
Then $(H_s,J_s)$ is a dissipative Floer datum interpolating between $(H_0,J_0)$ and $(H_1,J_1)$. Thus the moduli space of corresponding continuation solutions is compact. Moreover, since dissipativity is open with respect to uniform convergence, we perturb to obtain a regular moduli space required to define the continuation map.

Given a family of dissipative Floer data, we similarly have compactness in the family. Thus a standard construction in Floer theory produces a chain homotopy. It is also standard to obtain functoriality with respect to $\leq$.

To show independence of choices we need to show that any two dissipative Floer data can be connected by a family. This is an obvious variant of the construction of continuation Floer data which is carried out in \cite{Groman2015}. We do not need access to the details of this construction for our later purposes.

Finally, for the last statement, observe first that in the particular case $H_s=H_0+c(s)$ and $J_s=J$ the claim is straightforward. Let $C>H_1-H_0$. Considering this and the functorial maps associated to the inequalities
\[
H_0\leq H_1\leq H_0+C\leq H_1+C,
\]
we deduce that the natural maps are isomorphisms.
\end{proof}

A particular consequence of the last theorem is that $HF^*(H,J_1)=HF^*(H,J_2)$ provided $H$ is dissipative with respect to both $J_1$ and $J_2$. We therefore drop $J$ from the notation henceforth.

\subsection{Symplectic cohomology rings}
We now proceed to define operations of Floer theory associated to punctured spheres. We shall focus on the product. Denote by $\Sigma_{0,1,2}$ the thrice punctured sphere together with a choice of cylindrical coordinates near each puncture, identifying one of them with the half cylinder $\R_+\times S^1$ and the other two with the half cylinder $\R_-\times S^1$. We refer to these as positive and negative punctures respectively. For $i=1,2,3$ and  Floer data $(H_i,J_i)$ satisfying $H_3\geq H_1+H_2$, a product datum is a pair $(\mathfrak{H},J)$ where $J$ depends on $z\in\Sigma_{0,1,2}$, $\mathfrak{H}$ is a 1-form with values in the space of Hamiltonians such that the restriction near the $i$th puncture is the Floer datum $(H_idt,J_i)$ and such  that there is a constant $a$ for which the $M$-dependent $2$-form $d\mathfrak{H}_x+a\omega_{\Sigma_{0,1,2}}$ is positive. Writing
\[
\mathfrak{H}=Fdt+Gds,
\]
the last condition is just
\[
\partial_sF(x,s,t)-\partial_tG(x,s,t)> -a,\quad\forall x\in M.
\]

When trying to mimic the discussion of dissipative Floer data from  the previous section, we run into a new difficulty in that the Poisson bracket $\{\mathfrak{H},\mathfrak{H}\}=\{F,G\}ds\wedge dt$ does not have to vanish. If it is not bounded, then in general one cannot turn the corresponding connection 2-form into a non-degenerate form by adding the pullback of a $2$-form on the base. Here we avoid this issue by only considering the product when $H_i$ is a multiple of a given Hamiltonian $H$ outside a compact set. An interpolating product Floer datum is called dissipative if the condition of the previous section is satisfied mutatis mutandis.

A particular case is obtained by considering a closed $1$-form $\alpha$ on the pair of pants equal to $dt$ at each input and $2dt$ at the output.  We then consider the Floer datum $(H\alpha,J_s)$ for $J$ an appropriate domain-dependent Floer datum on the pair of pants. The corresponding Floer datum is automatically dissipative since it is locally the same as the standard Floer equation.  We can achieve regularity by perturbing within a compact set.

Assuming regularity, this defines an operation
\[
* \colon CF^*(H,J)\otimes CF^*(H,J)\to CF^*(2H,J),
\]
by counting solutions to the Floer equation

\begin{equation}\label{eqProdFloer}
(du-\alpha\otimes X_H)^{0,1}=0. 
\end{equation}
More precisely, given a pair $(\gamma_1,[w_1]),(\gamma_2,[w_2])$, a pair of pants $u$ with input $(\gamma_1,\gamma_2)$ and output $\gamma_3$ contributes the term 
\begin{equation}\label{eqProdWeight}
T^{E^{\operatorname{top}}(u)+I_{\sigma}(w_1)+I_{\sigma}(w_2)-I_{\sigma}(w_3)}(\gamma_3,w_3),
\end{equation}
where \begin{itemize}
\item $w_3$ is any choice of path from the base point of the loop space component containing $\gamma_3$ to $\gamma_3$. 
\item
$E^{\operatorname{top}}$ is the \emph{topological energy} of $u$ defined by
$$E^{\operatorname{top}}:=\int u^*\omega_{\sigma}+dH\wedge\alpha$$
\end{itemize}
The expression \eqref{eqProdWeight} is independent of this choice of $w_3$.

We now proceed to use this to define the symplectic cohomology ring. Define a pre-order relation on the set of Hamiltonians by 
$$H_1\preceq H_2\iff \exists C\in\R
 \mbox{ such that } H_1(x)\leq H_2(x)+C \quad\forall x\in M.$$

Let $(H_i,J_i)$ be a sequence of dissipative Floer data such that $H_i\preceq H_{i+1}$ and such that  for each $i$ there is $j$ so that $2H_i\preceq H_j$,  and such that $g_{J_i}$ is equivalent to a fixed Riemannian metric $g$. We define
\[
SH^*(M;\{H_i\}):=\varinjlim HF^*(H_i).
\]
More generally, let $\mathcal{H}$ be a family of Floer data $(H,J)$ of dissipative Floer such that for all $(H,J)\in\mathcal{H}$, $g_J$ is equivalent to a fixed metric and suppose $\mathcal{H}$ is closed with respect to addition.   Then we define $SH^*(M;\mathcal{H})$ by taking the direct limit over $\mathcal{H}$.

Now suppose that $\mathcal{H}$ contains a $\preceq$-cofinal sequence of dissipative Hamiltonians $(H_n,J_n)$ where $n$ ranges over the natural numbers and that for each $n$, the datum $(2H_n,J_n)$ is also dissipative.  Then we have
\[
\varinjlim_nHF^*(H_n,J_n)=SH^*(M;\mathcal{H})=\varinjlim_nHF^*(2H_n,J_n).
\]
Under this assumption we define a product $*$ on $SH^*(M;\mathcal{H})$ induced by $*$ as defined above. We then have the following two theorems. We refer the reader to \cite{Groman2015} for proofs.
\begin{thm}\label{tmProduct}
The operation $*$ commutes with the continuation maps on the level of cohomology. The induced operation on $SH^*(M;\mathcal{H})$ depends only on $\mathcal{H}$ and not on any of the other choices made in the construction. The operation is associative and super-commutative.
\end{thm}

\begin{thm}
If $\varepsilon>0$ is sufficiently small then there is an isomorphism $H^*(M ; \Lambda_{\operatorname{univ}})=HF^*(H_\varepsilon)$. This makes $SH^*(M;\mathcal{H})$ into a unital algebra over $H^*(M ; \Lambda_{\operatorname{univ}})$ equipped with the quantum product.
\end{thm}

\section{Floer homology of magnetic cotangent bundles}

\subsection{Dissipativity of linear Hamiltonians}
Let us now return to the case of $M = T^*N$ where $N$ is a closed orientable manifold. We will work with Hamiltonians which are radially linear at infinity in the sense of Subsection \ref{sec:radial},  that is, Hamiltonians $H$ of the form $H(p,q)=a|p|+b$ for some constants $a,b$ outside of a compact set. We shall consider also roughly linear Hamiltonians obtained from linear ones by a sufficiently $C^2$-small perturbation.

\begin{thm}\label{tmLinDiss}
Let $H$ be linear at infinity with slope $a$ which is not in the period spectrum of the (non-magnetic) geodesic flow, let $J$ be conical, and let $\sigma$ be a magnetic form. Then $(H,J)$ is dissipative for $\omega_\sigma$.
\end{thm}
We remark that the period spectrum of the geodesic flow is always a closed nowhere dense subset of $\mathbb{R}_+$ (see for instance \cite[Proposition 3.7]{Schwarz2000}).

\begin{proof}
We first prove the associated Gromov metric is uniformly isoperimetric. The latter property is implied by bounded geometry which for $\sigma=0$ is established in \cite[Example 5.19]{Groman2015}. Being uniformly isoperimetric is stable under metric equivalence. By Lemma~\ref{lmMagPertbDec} we have that $X_{H,\sigma}$ is arbitrarily close (with respect to a conical metric) to $X_H$. Moreover $\omega_\sigma$ becomes arbitrarily close to $\omega$. So the corresponding Gromov metrics are uniformly equivalent (in fact, the equivalence constant converges to unity).

It remains to show condition \ref{dfDess2} of Definition \ref{dfDess}. By Lemma~\ref{lmMagPertbDec} the time-1 magnetic flow becomes arbitrarily close to the geodesic flow.  By Remark \ref{remLyapunov} the geodesic flow satisfies the condition \ref{dfDess2} whenever the slope is not in its period spectrum.
\end{proof}
\begin{rem}
Note that it is \textit{essential} in the last argument that we are considering a \textit{linear} Hamiltonian. The time-$1$ magnetic orbit of a quadratic Hamiltonian bears no relation to the that of the ordinary geodesic flow.
\end{rem}

It follows from Theorem \ref{tmLinDiss} that if $\mathcal{H}_{\operatorname{lin}}$ is the collection of functions on $T^*N$ which, for some constants $a,b$ outside of a compact set coincide with the function $ae^r+b$ then the symplectic cohomology groups $ SH^*(T^*N ; \mathcal{H}_{\operatorname{lin}}  \colon \! \omega_{ \sigma})$ are well defined.
\subsection{The relation with Viterbo's symplectic cohomology}
First however let us clarify an important point. For this we temporarily denote the ``usual'' symplectic cohomology (as originally defined by Viterbo \cite{Viterbo1999}) of $T^*N$ with respect to the standard symplectic form $ \omega = - d \lambda$ by $SH_{\mathrm{Viterbo}}(T^*N \! \colon \!  \omega)$. A priori, it is not clear that our symplectic cohomology groups agree with $SH_{\mathrm{Viterbo}}(T^*N  \! \colon \! \omega)$. Nevertheless, this is indeed the case, as the following result shows.

\begin{thm}
Let $\mathcal{H}_{\operatorname{lin}}$ be the collection of functions on $T^*N$ which, for some constants $a,b$ outside of a compact set coincide with the function $ae^r+b$. Then
\[
SH^*(T^*N; \mathcal{H}_{\operatorname{lin}} \! \colon \! \omega)=SH_{\mathrm{Viterbo}}(T^*N  \! \colon \! \omega).
\]
\end{thm}
\begin{proof}
$SH_{\mathrm{Viterbo}}(T^*N  \! \colon \! \omega)$ is defined as a limit over Floer homologies defined using Hamiltonians from $\mathcal{H}_{\operatorname{lin}}$ and conical almost complex structures. These Floer data have been shown above to be dissipative, so to prove the claim it remains to show that the continuation maps on both sides coincide. In \cite{Viterbo1999} continuation maps are defined using families $H_s$ where $H_s=h_s(e^r)$ outside of a compact sets and $\partial_s(h'_s)\geq 0$. Call such a family admissible. Admissible families are not automatically dissipative, but to prove the claim it suffices to show that for any regular $H_1\leq H_2\in \mathcal{H}_{\operatorname{lin}}$ there is an admissible family which is also dissipative. For this, in the proof of Theorem \ref{tmContinuatiation} take the function $ \eta$ to be $e^{r/2}$ instead of the distance function to a point. The rest of the construction then creates a dissipative Floer datum which is also admissible.
\end{proof}

From now on whenever speaking of the symplectic cohomology of $T^*N$ (irrespective of the symplectic form) we implicitly use the semigroup $\mathcal{H}_{\operatorname{lin}}$, and thus omit it from our notation\footnote{This is now consistent with the notation from the Introduction.}.

\begin{rem}
Implicit in the notation $\mathcal{H}_{\operatorname{lin}}$ is a choice of a Riemannian metric $g$ on the base manifold $N$ (cf. the definition of linear Hamiltonians in Subsection \ref{sec:radial}). Let us momentarily make this dependence explicit and write $\mathcal{H}_{\operatorname{lin},g}$. Then the results of \cite{Groman2015} imply there is a canonical ring isomorphism 
\begin{equation}
SH^*(T^*N; \mathcal{H}_{\operatorname{lin},g_1} \! \colon \! \omega)=SH^*(T^*N; \mathcal{H}_{\operatorname{lin},g_2} \! \colon \! \omega)
\end{equation}
for any pair of Riemannian metrics $g_1,g_2$ on $N$. 
\end{rem}

\section{Deformation of Floer homology}
\label{sec:deformations}
\subsection{The extended Floer-Novikov complex}
\label{sec:extend}

As in Section \ref{sec:covers}, let $ \Gamma$ denote the image in $H_2(N ; \mathbb{Z}) \big/ \mathrm{Tor}$ represented by maps from spheres and tori. Suppose that $ \mathcal{K}$ is a (possibly degenerate) polytope in $H^2(N; \mathbb{R})$. We now define various geometric objects related to $\mathcal{K}$. The reader is forewarned that the notation will eventually get quite hairy -- hence we take it slow and introduce things step-by-step.

\begin{df}[\textbf{The group $\Gamma_\mathcal{K}$}]
The collection $  \left\{ I_k \mid k \in \mathcal{K} \right\} $ determines a polytope $ \Phi_{ \mathcal{K}} \subset H^1( \Gamma ; \mathbb{R})$. We abbreviate the quotient of $ \Gamma/\bigcap_{\phi \in \Phi_{\K}} \ker \phi$ by   $ \Gamma_{ \mathcal{K}}$. 
\end{df}

\begin{df}[\textbf{The ring $\Lambda_{\mathcal{K}}$}]
To each polytope $K\subset H^1( \Gamma ; \mathbb{R})$ we associate a ring $\Lambda_{\cK}$ defined as the set of formal sums
$$
\sum_{ \gamma \in \Gamma_{\cK}} c_{ \gamma} z^{ \gamma}
$$
where $c_{\gamma}\in\Lambda_{\operatorname{univ}}$, $z^{ \gamma}$ is formal, and  the terms $ c_{ \gamma} z^{ \gamma}$ satisfy the \textit{Novikov multi-finiteness condition}: 
\vspace{0.75em}

\begin{displayquote}
For all  $c >0$  and $\phi \in \Phi_{\K}$  the set 
$$
\left\{ \gamma \in \Gamma \mid c_{ \gamma} \ne 0, \ \phi(\gamma)+\val(r_{\gamma}) < c \right\} 
$$
is finite.
\end{displayquote}
\end{df}
\begin{rem}
The assignment $\cK\mapsto \Lambda_\cK$ is contravariantly functorial. 

We denote by $\cO$ the pre-sheaf $\cK\mapsto\Lambda_\cK$.
\end{rem}  

\begin{df}[\textbf{The covering spaces $\mathcal{L}_KT^*N$ and $\mathcal{L}_KN$}]
Let us denote by $ \mathcal{L}_{ \mathcal{K}}N$ and $ \mathcal{L}_{ \mathcal{K}}T^*N$ the abelian covers of $N$ and $T^*N$ with deck transformation group $ \Gamma_{ \mathcal{K}}$. By construction, these covers have the property that if $ [ \sigma] \in \mathcal{K}$ then the transgressed form $ \tau( \sigma)$ pulls back to an exact form on $ \mathcal{L}_{ \mathcal{K}}N$.
\end{df}

\begin{df}[\textbf{The module $V_{\mathcal{K},\sigma}$}]
Let $H$ be a Hamiltonian on $T^*N$. Denote by 
$$
\mathcal{P}_{\K}(H,\sigma)\subset \mathcal{L}_{ \mathcal{K}}T^*N
$$
the set of lifts of the periodic orbits of $H$ with respect to $\omega_{ \sigma}$. The action of $\Gamma_{\mathcal{K}}$ restricts to an action on $\mathcal{P}_{\K}(H,\sigma)$. We define
$$
V_{\cK,\sigma}\coloneqq  \mathbb{K}\langle \mathcal{P}_{\K}(H,\sigma)\rangle.
$$
This is a finitely generated free module over the group ring $\mathbb{K}[\Gamma_{\K}]$.
\end{df}
We can now start on the definitions of the relevant Floer complexes.
\begin{df}[\textbf{The Floer complex}]
Let $\sigma \in \Omega^2(N)$ be a closed two-form. We do \textbf{not} necessarily assume that $\sigma$ belongs to the polytope $\mathcal{K}$. Suppose $(H,J)$ is strongly dissipative with respect to some $\omega_{ \sigma}$. We define the \emph{extended Floer chain complex}\footnote{We should really include both the almost complex structure $J$ and the coefficient field $\mathbb{K}$ in our notation. This would, however, push the notation from being merely heavy to being utterly unreadable, and hence we hope the reader forgives us this little imprecision.} 
$$
 CF(H  \! \colon \! \mathcal{K}, \omega_\sigma ) 
 \coloneqq V_{\cK,\sigma}\otimes_{\mathbb{K}[\Gamma_{\K}]}\Lambda_{\cK}.
$$
Explicitly, elements of $ CF(H  \! \colon \! \mathcal{K}, \omega_\sigma )$ can be presented as possibly infinite sums $\sum c_i\otimes\gamma_i$ for $c_i\in\Lambda_{\operatorname{univ}}$ and  $\gamma_i\in \mathcal{P}_{\K}(H,\sigma)$, satisfying the following multi-finiteness condition: 
\vspace{0.75em}

\begin{displayquote}
Each $\sigma'\in\K$ gives rise to an action filtration $\mathcal{A}_{H,\omega_{ \sigma'}}$ on $V_{\cK,\sigma}$. We require that for any $\sigma'\in\K$ we have $\lim_{i\to\infty} \mathcal{A}_{H,\omega_{ \sigma'}}(\gamma_i)+\val(c_i)=-\infty.$ 
\end{displayquote}
\end{df}

\begin{rem}
For fixed choice of $\sigma$, the assignment $\cK\mapsto  CF(H  \! \colon \! \mathcal{K}, \omega_\sigma )$ is contravariantly functorial as a map to $\Lambda_{\operatorname{univ}}$-modules, and defines a presheaf of $\cO$-modules. 
\end{rem}

We would like to define a linear operator $d_{\sigma}= d_{H,J, \omega_\sigma}$ on $CF(H  \! \colon \! \mathcal{K}, \omega_\sigma )$ by counting Floer trajectories in the usual way. The Floer trajectories are defined with respect to the Hamiltonian flow determined by $\omega_{ \sigma}$. The only difference to the ordinary Floer complex is in the multi-finiteness condition determined by $\mathcal{K}$. In general, however, the differential $d_\sigma$ is only a formal sum of operators on $CF(H  \! \colon \! \mathcal{K}, \omega_\sigma )$ and is \textit{not} well defined as an operator from  $CF(H  \! \colon \! \mathcal{K}, \omega_\sigma )$ to itself. That is, having $x\in CF(H  \! \colon \! \mathcal{K}, \omega_\sigma )$  does not guarantee that the formal sum $d_{\sigma}(x)$ satisfies the multi-finiteness condition.  Our main result, Theorem \ref{tmMainIso}, gives conditions when it is. We will continue to refer to the $\Lambda_{\K}$-module $CF(H  \! \colon \! \mathcal{K}, \omega_\sigma )$ as the Floer complex, even though it only has a formal differential. \\

Given a pair of Floer data $(H_1,J_1)\preceq (H_2,J_2)$ and a dissipative Floer datum $\mathfrak{H}$ interpolating between them, an analogous discussion applies to the continuation map
$$
CF(H_1  \! \colon \! \mathcal{K}, \omega_\sigma )\to CF(H_2  \! \colon \! \mathcal{K}, \omega_\sigma ).
$$
Namely, each Floer solution $u$ defines a $\Lambda_{\K}$-module map $C_u$ but the formal sum $\Sigma_uC_u$ is not guaranteed to converge. 

Finally, \emph{if we restrict attention to the case where $\cK$ consists of a single point $\sigma'$}, which may or may not equal $\sigma$, we can similarly formally define the pair of pants product. Namely, we assign to each solution $u$ to equation \eqref{eqProdFloer} defined by the Hamiltonian flow with respect to $\omega_{\sigma}$ the operation
$$
P_u:CF(mH  \! \colon \! \{\sigma'\}, \omega_\sigma )\otimes CF(nH  \! \colon \! \{\sigma'\}, \omega_\sigma )\to CF((m+n)H  \! \colon \! \{\sigma'\}, \omega_\sigma )
$$
described by replacing $\sigma$ with $\sigma'$ in \eqref{eqProdWeight}. As in the case of the differential, the sum  $\Sigma _u P_u$ of all such solutions $u$ doesn't necessarily converge to an operator.

We are nearly ready to state our main result. But first, we need one further notational extension that allows us to speak of varying the polytope $\mathcal{K}$.

\begin{df}[\textbf{The polytope $\mathcal{Q}$}]
\label{def-Q}
Let $\mathcal{Q} \subset H^2(N; \mathbb{R})$ be a fixed ``master'' polytope. Choose a section 
\[
\sigma \colon \mathcal{Q} \to \Omega^2(N), \qquad k \mapsto \sigma_k.
\]
We then choose:
\begin{enumerate}[label=(\roman*)]
  \item A smaller (possibly degenerate) polytope $\mathcal{K} \subset \mathcal{Q}$;
  \item A fixed element $k \in \mathcal{Q}$. It is \textbf{not} required that $k \in \mathcal{K}$.
\end{enumerate}
This gives us the Floer complex $CF(H \! \colon \! \mathcal{K},\omega_{\sigma_k})$. In order to indicate the dependence on $\mathcal{Q}$ we sometimes write this as 
$$
CF(H \! \colon \! \mathcal{Q},\mathcal{K},\omega_{\sigma_k}).
$$
\end{df}

In fact, all of our applications require only a special case of Definition \ref{def-Q}, where $\mathcal{Q}$ is an interval. To help the reader wade through our notation, we spell this case out explicitly:

\begin{ex}[\textbf{Two ends of an interval}]
\label{ex-interval}
Fix $\sigma_0,\sigma_1 \in H^2(N; \mathbb{R})$ and let $\sigma_r := (1-r)\sigma_0 + r \sigma_1$.

\begin{equation}
\label{eq-Q-interval}
\mathcal{I} = \left\{ [\sigma_r] \mid 0\le r \le 1 \right\},
\end{equation}

and define $\sigma$ to be the obvious section
$$
\sigma \colon \mathcal{I} \to \Omega^2(N), \qquad r \mapsto \sigma_r.
$$
We then take
$$
\mathcal{K} = \{ [ \sigma_0] \}.
$$
Fix now $r \in [0,1]$ and consider the complex $CF(H \! \colon \! \mathcal{I} , \mathcal{K} , [\sigma_r] )$. The role of $\sigma_0$ is to single out a cover of the loop space and the action of the corresponding group of deck transformations on $\Lambda_{\operatorname{univ}}$. The role of $\sigma_r$ is to define the Hamiltonian flow. Thus the complex is generated by the $\omega_{\sigma_r}$-periodic orbits and the Floer trajectories are lifts of solutions to Floer's equation with respect to $\omega_{\sigma_r}$. Note in particular if $r = 0$ then 
\begin{equation}
  \label{eq-r0}
  CF^*(H,J \! \colon \! \mathcal{I}, \mathcal{K}, [\sigma_0]) = CF^*(H,J \! \colon \! \sigma)
\end{equation}
\end{ex}

\subsection{The main theorem}
In this section we state the main result of the present paper. The proof is deferred to Section \ref{sec:proof_main}.

\begin{thm}\label{tmMainIso}
Let $\mathcal{Q} \subset H^2(N;\R)$ be a possibly degenerate polytope. Let $\sigma \colon \mathcal{Q} \to\Omega^2(N)$ be a smooth section. Then there exists an assignment $ \K \mapsto \mathcal{H}_{\K}$, for $\K \subset \mathcal{Q}$ a possibly degenerate polytope, of a set  of Floer data $(H,J)$ such that $H$ is linear at infinity and $J$ is conical, which satisfy the following properties:
\begin{enumerate}
\item  \label{tmMainIsoA} For each $\K \subset \mathcal{Q}$ the set $\mathcal{H}_{\K}$ is $\preceq$-cofinal  in the set $\mathcal{H}_{\operatorname{lin}}$ of all Floer data that are linear at infinity.
\item  \label{tmMainIsoB} For any inclusion $\K' \subset\K $ we have $\mathcal{H}_{\K}\subset \mathcal{H}_{\K'}$.
\item  \label{tmMainIsoC} $\mathcal{H}_{\K}$ is open in the uniform convergence topology on the space of time-dependent Hamiltonians.
\item  \label{tmMainIsoD} For each $\mathcal{K}' \subset \mathcal{K} \subset \mathcal{Q}$ and for each $k \in \mathcal{K}$ the Floer differential $d_{\omega_k}$ is well-defined on the complex $CF(H \! \colon \! \mathcal{K}, \mathcal{K}', \sigma_k)$. 
\item \label{tmMainIsoE} For each $(H,J)\in\mathcal{H}_{\K}$ there exists a chain homotopy equivalence
\[
  f_{k_1k_2} \colon CF(H \! \colon \! \K,k_1)\to CF(H \! \colon \! \K,k_2)
\]
for any choice $k_1,k_2\in \mathcal{Q}$. 
Moreover, the map $f$ is well defined upon restricting $\K\to\K'$ for any  $\K'\subset \mathcal{K}$ and induces a chain homotopy equivalence
\[
  CF(H \! \colon \! \K',k_1)\to CF( H \! \colon \! \K',k_2).
\]
\item\label {tmMainIsoF} 
Given a pair $H_1\preceq H_2\in\mathcal{H}_{\K}$ there exist dissipative continuation data  for which the corresponding continuation map is well defined. They induce maps on the homologies 
$$HF^*(H_1 \! \colon \! \mathcal{K}, \mathcal{K}', \sigma_k)\to  HF^*(H_2 \! \colon \! \mathcal{K}, \mathcal{K}', \sigma_k).$$ These continuation maps fit into a commutative diagram

$$
\xymatrix{HF^*(H_1 \! \colon \! \mathcal{K}, \mathcal{K}', \sigma_k)\ar[d]^{f_{k_1,k_2,*}} \ar[r] & HF^*(H_2 \! \colon \! \mathcal{K}, \mathcal{K}', \sigma_k)\ar[d]^{f_{k_1,k_2,*}}\\
HF^*(H_1 \! \colon \! \mathcal{K}, \mathcal{K}', \sigma_k)\ar[r]&HF^*(H_2 \! \colon \! \mathcal{K}, \mathcal{K}', \sigma_k).}
$$
\item \label{tmMainIsoG}  Restricting attention to the case where $\K=\{k'\}$ is a singleton there is a well defined pair of pants product at the homology
$$HF^*(H \! \colon \! \mathcal{K}, \{k'\}, \sigma_k)\otimes HF^*(H \! \colon \! \mathcal{K}, \{k'\}, \sigma_k) \to HF^*(H' \! \colon \! \mathcal{K}, \{k'\}, \sigma_k)$$
where $H'$ is any Hamiltonian in $\cH_{\K}$ such that $2H\preceq H'$. 
The maps $f_{k_1,k_2,*}$ on homology commute with this product.
\end{enumerate}
\end{thm}

Part \ref{tmMainIsoE} of the Theorem states that \emph{for the Hamiltonian $(H,J)\in\cH_{K}$ there Floer cohomology with respect to $\omega_0$ can be computed using Floer's equation for $\omega_1$}. In particular, this says that the complexes $CF(H \! \colon \! \mathcal{I} , [ \sigma_0] , [\sigma_r] )$ from Example \ref{ex-interval} satisfy:
\vspace{0.75em}

\begin{displayquote}
The complexes $CF(H \! \colon \! \mathcal{I} , [ \sigma_0] , [\sigma_r] )$ are, up to chain \\  homotopy equivalence, independent of $r$.
\end{displayquote}
\vspace{0.75em}

Note this only holds for Hamiltonians in $\cH_{\mathcal{I}}$. In general we do not expect such an isomorphism for arbitrary Hamiltonians.

\subsection{Twisted symplectic cohomology}

We recall the definition of symplectic cohomology twisted by a transgressed form from \cite{Ritter2013}. \emph{In this section we deal exclusively with the twisted symplectic cohomology of the Liouville form $\omega_0$.} This is all we shall need in order to make contact with Ritter's result \cite{Ritter2009}. Moreover, this will allow us to not worry about convergence issues. 
For each $\sigma\in H_2(T^*N)$ consider the local system $\mathscr{C}_{\tau(\sigma)}$ introduced in Section \ref{SecCoefficients}. Given a Floer datum $(H,J)$ with $H$ linear at infinity the the underlying module of the $\mathscr{C}_{\tau(\sigma)}$-twisted Floer complex is 
\begin{equation}
CF^*(H,J \! \colon \! \omega_0 \,;\mathscr{C}_{\tau(\sigma)})=\Lambda_{\operatorname{univ}}\langle\mathcal{P}(H)\rangle.
\end{equation}
That is, the free $\Lambda_{\operatorname{univ}}$-module generated by the periodic orbits of $H$ graded by the Conley-Zehnder index. A Floer trajectory $u$ of index difference $1$ going from a periodic orbit $\gamma_1$ to a periodic orbit $\gamma_2$ contributes the term $T^{\tau_{\sigma}(u)}\gamma_2$ to $\gamma_1$. Observe that since we are considering Floer's equation for the Liouville form \emph{the set of all Floer trajectories emanating from a given periodic orbit $\gamma_1$ is finite, regardless of any additional energy assumptions.} Thus, the Floer differential is automatically well defined. 

It is shown in  \cite{Ritter2013} that
\begin{itemize}
\item 
The twisted differerential squares to $0$. 
\item Defining twisted continuation maps similarly by weighting the continuation trajectories by $T^{\alpha_{\sigma}(u)}$, the resulting maps are chain maps, and the induced maps on homology are functorial and independent of choice of continuation datum.
\item The twisted pair of pants product defined  by weighting each pair of pants $u$ by $T^{\alpha_{\sigma}(u)}$ satisfies the Leibnitz rule. 
\end{itemize} 
This gives rise to twisted symplectic cohomology
$$
SH^*(T^*N)_{\tau(\sigma)} \coloneqq \varinjlim HF^*(H \! \colon \! \omega_0 ; \mathscr{C}_{\tau(\sigma)})
$$
where the limit is taken over $\mathcal{H}_{\operatorname{lin}}$. We then have the following result of Ritter.
\begin{thm}[Ritter \cite{Ritter2009}]\label{tmRitter}
$SH^*(T^*N)_{\tau(\sigma)}=H_{n-*}(\mathcal{L}N)_{\tau(\sigma)\otimes\tau(w_2)}$
\end{thm}

In order to combine Ritter's Theorem with our main Theorem to prove Theorem \ref{thm:main} from the introduction we shall need the following variant of Eilenberg's Theorem \cite[Theorem VI.3.4]{Whitehead1978}.
\begin{pr}\label{prEilenberg}
Let $\mathcal{I}$ be the interval with endpoints $\sigma_0 = \sigma$ and $\sigma_1 = 0$, as in Example \ref{ex-interval}. Let $(H,J) \in \mathcal{H}_{\mathcal{I}}$. Then there is an a isomorphism of chain complexes over $\Lambda_{\operatorname{univ}}$
\begin{equation}
\label{eq-IC}
CF(H,J \! \colon \! \mathcal{I},[\sigma],[0]) \cong CF^*(H,J \! \colon \! \omega_0 \,;\mathscr{C}_{\tau(\sigma)}). 
\end{equation}
Given a pair $(H_1,J_1)\leq (H_2,J_2)$ and a dissipative homotopy from $(H_1,J_1)\to(H_2,J_2)$, the induced continuation map is intertwined by the isomorphism of \eqref{eq-IC}. Similarly, given a choice of Floer data for the pair of pants product and the unit, the corresponding operations are intertwined by the  isomorphism of \eqref{eq-IC}.
\end{pr}
\begin{proof}

For the readers' convenience we unwind the definitions which are scattered throughout this section. As a $\Lambda_{univ}$-module, the left hand side is 
$$\mathbb{K}\left\langle \mathcal{P}_{[\sigma_1]}(H,\omega_0)\right\rangle\otimes_{\Gamma_{[\sigma_1]}}\Lambda_{[\sigma_1]}$$
where:
\begin{itemize}
\item $\mathcal{L}_{[\sigma_1]}T^*N$ is the abelian cover of $\mathcal{L}T^*N$ corresponding to the $2$-form $\sigma$ (which, by transgression, is a $1$-form on the loop space). 
\item   $\mathcal{P}_{[\sigma_1]}$ is the lift of the set of $1$- periodic orbits of  the Hamiltonian $H$ \emph{with respect to $\omega_0$} to the covering space $\mathcal{L}_{[\sigma_1]}T^*N$.
 \item $\Gamma_{[\sigma_1]}$ is the group of deck transformations of the cover $\mathcal{L}_{[\sigma_1]}T^*N$.
 \item $\Lambda_{[\sigma_1]}$ is the completion of the group algebra of $\Lambda_{univ}[\Gamma_{[\sigma_1]}]$ with respect to the filtration induced by combining the valuation of $\Lambda_{univ}$ with the valuation  induced by $\sigma_1$ on $[\Gamma_{[\sigma_1]}]$. 
\end{itemize} 

From this description it follows that the left hand side is generated over the universal Novikov ring by elements $\tilde{\gamma}:=(\gamma, [u])$ consisting of 
\begin{itemize}
\item a periodic orbit $\gamma$ of the Hamiltonian $H$ \emph{with respect to $\omega$}, 
\item an equivalence class $[u]$ of paths from the base loop $\gamma_0$ to $\gamma$. The equivalence relation defined by modding out by the kernel of $I_{\sigma_1}$. 
\end{itemize}
These generators are subject to the relation 
$$(\gamma, [u\#w])\sim T^{\tau_\sigma(w)}(\gamma,[u])$$
for $w$ any closed loop in the loop-space. 

The right hand side has been described in detail in this subsection.

We define a map from the left hand side to the right hand side by mapping $\psi:T^r\tilde{\gamma}=T^r(\gamma, [v])\mapsto T^{r+\tau_{\sigma}(v)}\gamma$. This map is well defined and commutes with the action of $\Lambda_{\operatorname{univ}}$ on each side.  Moreover we can invert $\psi$ as follows. Chose arbitrarily for each $\gamma\in\mathcal{P}_{\omega}(H)$ a homotopy class $[u_0]$ to $\gamma$ from the base path in the homotopy class of $\gamma$. Then define the map $\phi$ from the left hand side to the right hand side by taking $T^r\gamma\mapsto T^{r-\tau_{\sigma}(v_0)}(\gamma,[v_0])$. The map $\phi$, which is independent of all choices, is clearly inverse to $\psi$ from both the left and right.  

The map $\psi$ is a chain map. Indeed, given a Floer trajectory $u$ from a periodic orbit $\gamma_0$ to a periodic orbit $\gamma_1$ it acts on the left by mapping an element $(\gamma_0, [w])$ to $(\gamma_1,[w\#u])$, while on the right it acts by mapping $\psi(\gamma_0,[w])=T^{\tau_{\sigma}(w)}\gamma_0$ to $T^{\tau_{\sigma}(u)+\tau_{\sigma}(w)}\gamma_1=\psi(\gamma_1,[w\#u])$. 

That $\psi$ intertwines the continuation map follows in the same way. That $\psi$ intertwines the product follows from the definition of the product in \eqref{eqProdWeight}.

\end{proof}

\begin{rem}
The careful reader will have noted that Part \ref{tmMainIsoF} of Theorem \ref{tmMainIso} does not contain a statement concerning functoriality of the continuation maps and independence of the choice of dissipative datum. While these are certainly true, a detailed proof is rather tedious to spell out. For our main application the functoriality and independence follow by the discussion preceding Theorem \ref{tmRitter} and the last proposition. 
\end{rem}

Combining all the above we have

\begin{thm}\label{tmMain}
For $H\in\mathcal{H}_\K$ and $\omega$ the canonical symplectic form we have
\begin{equation}\label{eqtmMain1}
HF^*(H \!  \colon \! \omega_\sigma) \cong HF^*(H \! \colon \! \omega \,;\mathscr{C}_{\tau(\sigma)})
\end{equation}
This map is natural with respect to continuation maps and commutes with the product. 
In particular, we have isomorphisms of rings
\begin{equation}\label{eqtmMain2}
SH^*(T^*N  \! \colon \! \omega_{\sigma})=SH^*(T^*N)_{\tau(\sigma)}.
\end{equation}
\end{thm} 
\begin{proof}

We apply Example \ref{ex-interval} with $\sigma_0 = \sigma$ and $\sigma_1 = 0$. As already remarked, part \ref{tmMainIsoE} of Theorem \ref{tmMainIso} produces a chain homotopy equivalence
$$
CF^*(H,J \! \colon \! \sigma) \stackrel{\eqref{eq-r0}}{\cong}  CF(H,J \! \colon \! \mathcal{I},[\sigma],[0]).
$$
This produces an isomorphism 
\begin{equation}
HF^*(H \! \colon \! \omega_{\sigma})\simeq HF^*(H \! \colon \! \mathcal{I}, [\sigma],[0]).
\end{equation}
The latter isomorphism commutes with continuation maps in $\cH_{\cK}$ and the pair of pants product according to Parts \ref{tmMainIsoF} and \ref{tmMainIsoG} of Theorem \ref{tmMainIso}. 
On the other hand, according to Proposition \ref{prEilenberg} we have 
$$
HF(H \! \colon \! \mathcal{I},[\sigma],[0]) \simeq HF^*(H \! \colon \! \omega;\mathscr{C}_{\tau(\sigma)})
$$
which again commutes with continuation maps and the Floer theoretic operations. 
Combining the two gives the isomorphism \eqref{eqtmMain1} and shows that it is natural with respect to continuation maps and commutes with operations.

To prove \eqref{eqtmMain2}, we rely on parts \ref{tmMainIsoA} and \ref{tmMainIsoC} of Theorem \ref{tmMainIso}  to pick a sequence $(H_i,J_i)\in\cH_{\mathcal{I}}$ which consists of regular Floer data and which is cofinal in the set $\mathcal{H}_{\operatorname{lin}}$ of all Hamiltonians that are linear at infinity. By naturality of  \eqref{eqtmMain1} we obtain the isomorphism of $\Lambda_{\operatorname{univ}}$-algebras
\begin{align}
SH^*(T^*N \! \colon \! \omega_{\sigma})&=\varinjlim HF^*(H_i  \colon \! \omega_{\sigma})\notag\\
&= \varinjlim HF^*(H_i  \colon \! \omega,\mathscr{C}_{\tau(\sigma)}) \notag\\
&=SH^*(T^*N)_{\tau(\sigma)}.\notag \qedhere
\end{align}

\end{proof}
\begin{proof}[Proof of Theorem \ref{thm:main}]
Theorem \ref{thm:main} is a consequence of Theorems \ref{tmRitter} and \ref{tmMain}.
\end{proof}

\section{Proof of the main theorem}
\label{sec:proofs}
\subsection{Continuation trajectories fixing $H$ and varying $\sigma$}
Consider a family $(\sigma_s)_{s\in[0,1]}$ of magnetic forms connecting a pair of magnetic forms $\sigma_0,\sigma_1$. 
We extend this family to all of $\mathbb{R}$ by setting $\sigma_s = \sigma_0$ for $s <0$ and $ \sigma_s = \sigma_1$ for $s>1$. Let $\sigma'_s \coloneqq \frac{\partial}{\partial r}\big|_{r = s} \sigma_r$.

Fix a critical point $[\gamma^i,w^i]$ of the functional $\mathcal{A}_{H,\sigma_i}$ for $i=0,1$ respectively (see \eqref{eqActionDef}), and fix $J = \{ J_{s,t} \}_{(s,t)\in\R\times S^1}$ such that $J_{s,t}$ is independent of $s$ for $|s|\gg 1$ and $J_{s,t}$ is compatible with $\omega_{\sigma_0}$ for $s$ near $-\infty$ and with $\omega_{\sigma_1}$ near $\infty$.
We denote by
\[
\mathcal{M}([\gamma^0,w^0],[\gamma^1,w^1];H, (\sigma_s)_{s\in [0,1]} ,J)
\]
the set of smooth maps $u:\R\times S^1 \to T^*M$
which satisfy the parametrised Floer equation
$$
\partial_{s}u+J_{s,t}(u)(\partial_{t}u-X_{H,\sigma_s}(u,t))=0,
$$
submit to the asymptotic conditions
$$
\lim_{s\to-\infty}u(s,t)=\gamma^0(t), \quad \lim_{s\to\infty}u(s,t)=\gamma^1(t),\quad \lim_{s\to\pm\infty}\partial_{s}u(s,t)=0,
$$
and such that $w^1=u\#w^0$. The energy of $u$ is defined via the $s$-dependent metric determined by $\omega_{\sigma_s}$ and $J_{\sigma_s}$. Namely,
\begin{equation}\label{eqsdepEnergy}
E(u) = E_{H,J}(u):=\int_{-\infty}^{\infty} \|\partial_{s}u\|_{\omega_{\sigma_s},J_{\sigma_s}}^2\,dt.
\end{equation}
 By a standard identity for time-dependent gradient flow, we have the following equation for any $u\in \mathcal{M}([\gamma^0,w^0],[\gamma^1,w^1];H, (\sigma_s)_{s\in [0,1]} ,J)$:
\begin{equation}\label{eqActionEnergyIdentity}
E(u) =\mathcal{A}_{H,\sigma_0}([\gamma^0,w^0])-\mathcal{A}_{H,\sigma_1}([\gamma^1,w^1])-\int_{-\infty}^{\infty}\left(\frac{\partial}{\partial s}\mathcal{A}_{H,\sigma_s}\right)(u(s))ds.
\end{equation}
Write $w^s:=w^0\#(u|_{(-\infty,s]\times S^1})$. The last term in the energy identity can be rewritten as
\begin{equation}\label{eq_error_term}
\int_{-\infty}^{\infty}\left(\frac{\partial}{\partial s}\mathcal{A}_{H,\sigma_s}\right)(u(s))ds=\int_{-\infty}^{\infty}\left(\int{(\pi\circ w^s)}^*\sigma_s' \right)ds
\end{equation}
Control of the last expression is the key to establishing compactness of the moduli spaces and thus the well-definedness of continuation maps. But first, we must show that energy control indeed guarantees compactness.
\begin{thm}\label{tmEneCompCont}
Let $J_s$ be roughly conical at infinity and let $(H,J_i)$ be strongly dissipative for both $\omega_{\sigma_0}$ and $\omega_{\sigma_1}$. For any $E>0$ the moduli space  
$$
\mathcal{M}_E([\gamma^0,w^0],[\gamma^1,w^1];H, (\sigma_s)_{s\in [0,1]} ,J)
$$
of continuation trajectories of energy $\leq E$ is compact in the sense of Gromov-Floer.
\end{thm}
\begin{proof}
We first show there is a compact set $K_E\subset T^*M$ such that any continuation trajectory
$$
u\in \mathcal{M}_E([\gamma^0,w^0],[\gamma^1,w^1];H, (\sigma_s)_{s\in [0,1]} ,J)
$$
 maps into $K_E$.

Each component of the complement of a slight thickening of $[0,1]\times S^1$ is contained in an a priori compact set by strong dissipativity. See Remark \ref{rmRobust}. Let $\tilde{J}$  denote the almost complex structure on $M\times [-\epsilon,1+\epsilon]\times S^1$ given at a point $(x,s,t)$ by $J_{X_H,\sigma_s}$. The latter is the almost complex structure produced by applying the Gromov trick to the almost complex structure $J_s$ and the vector field $X_{H,\sigma_s}$. Denote by $\tilde{\omega}_H$ the \emph{non-closed} $2$-form given at $(x,s,t)$ by $\omega_{\sigma_s}+dH\wedge dt+ds\wedge dt$. Then $\tilde{J}$ is compatible with $\tilde{\omega}_H$. For any compact measurable subset $S\subset \R\times S^1$ equation \eqref{eqEnergyIdArea} produces the relation
\begin{equation}
E_{\tilde{\omega}_H}(\tilde{u};S ):=\int_S\tilde{u}^*\tilde{\omega}_H=E_{H,J}(u)+ \operatorname{Area}(S).
\end{equation} 
Indeed, the derivation of \eqref{eqEnergyIdArea} is pointwise and does not use closedness of the symplectic form.

Outside of a large enough compact set we have that $|\omega_{\sigma_{s}}-\omega_{\sigma_{0}}|$ is arbitrarily small. Thus the \emph{closed} $2$-form $\tilde{\omega}':=\omega_{\sigma_{0}}+ds\wedge dt$ tames $\tilde{J}$ there. The metric $g_{\tilde{\omega}',\tilde{J}}$ determined by $\omega_{\sigma_0}+ds\wedge dt$ and $\tilde{J}$ is equivalent on the one hand to the metric determined by $\omega_{\sigma_0}$ and $J_{X_H,\sigma_0}$ and on the other hand to the metric determined by $\tilde{\omega}_H$ and $\tilde{J}$. 

It follows that $g_{\tilde{J}}$ is uniformly isoperimetric. Then, denoting by $S\subset \R\times S^1$ the region $S=[-\epsilon,1+\epsilon]\times S^1\cap u^{-1}(M\setminus K)$, an estimate on $E_{\tilde{\omega}_H}(\tilde{u};S )$ produces an estimate  on $E_{\tilde{\omega}'}(u)$.  The  monotonicity argument thus implies an a priori estimate on $u([-\epsilon,1+\epsilon]\times S^1)$.

Having established that $u$ is contained inside an a priori compact set, we proceed to establish Gromov compactness for solutions contained in a compact set $K$. For this we remind first that that Gromov compactness is domain local. See  \cite[\S4.6]{McDuffSalamon2004}. On the components of the complement of $[0,1]\times S^1\subset\R\times S^1$ we have the standard Floer equation, and so Gromov-Floer compactness holds there as usual. To conclude, it suffices to show that for $\epsilon>0$ small enough, Gromov compactness holds for subsets $(s-\epsilon,s+\epsilon)\times S^1$ for all $s\in [0,1]$. Now, for fixed $\epsilon>0$ small enough, we have that for each $s$, the form $\tilde{\omega}_{H,\sigma_s}$ tames the almost complex structure $J_{H,\sigma_{s'}}$ for all $s'\in (s-\epsilon,s+\epsilon)$ \emph{everywhere in $T^*N$}.  Applying the Gromov trick and the discussion of energy of the previous paragraph we have that the set of all restrictions to $(s-\epsilon,s+\epsilon)\times S^1$ of all $u\in  \mathcal{M}_E$ is a set of $\{J_{H,\sigma_{s'}}\}_{s'\in (s-\epsilon,s+\epsilon)}$ holomorphic maps with a priori energy bound. We have found a finite open cover of the domain $\R\times S^1$ so that compactness holds for the restriction of solutions to each element of the cover. Compactness for $\mathcal{M}_E$ follows.

\end{proof}
\subsection{A priori energy estimates}
We now specialise to the following case. Let $f \colon  \R\to[0,1]$ be smooth and satisfy $f(s)=1$ for $s\geq1$ and $f(s)=0$ for $s\leq0$. Let  
\begin{equation}\label{eqMagneticFormInterpolation}
\sigma_s=f(s)\sigma+\sigma_0,
\end{equation}
and let $J_{\sigma_s}$ be a family of $\omega_{\sigma_s}$-compatible almost complex structures which are roughly conical at infinity.  The right hand of \eqref{eq_error_term} specialises to
\[
e=\int_{-\infty}^{\infty}f'(s)\left(\int{w^s}^*(\pi^*\sigma) \right)ds.
\]
Write $u^s:=u|_{(-\infty,s)\times S^1}$. Then we have
\begin{align}\label{eqActnEst}
|e|&\leq \int_0^1|f'(s)|\left(\left|\int w^{0*}(\pi^*\sigma)\right|+\left|\int u^{s*}(\pi^*\sigma)\right|\right)ds\\
&\leq \left|\int w^{0*}(\pi^*\sigma)\right|+\sup_{s\in[0,1]}\left|\int_ {(-\infty,s)\times S^1}u^*(\pi^*\sigma)\right|.\notag
\end{align}
The first term in the rightmost expression may seem strange at first sight, but it accounts for the fact that the action difference is not invariant under shifts of the base loop. This will turn out to be harmless as Floer's equation is invariant under such shifts. We need to control the rightmost term. In general this cannot be achieved. However, we carefully construct  a family of Floer data which is  $\preceq$-cofinal in the set of Hamiltonians of the form $a|p|+b=ae^r+b$ for which we have such control is possible. The construction is as follows.

A choice of Riemannian metric $g$ on $N$  fixes a radial coordinate on $T^*N$. Let $\ell_0$ be the period of the shortest non-magnetic geodesic on $(T^*N,g)$.  For any  $a\in\R_+$ in the complement of the period spectrum of the geodesic flow, let $h_a:\R_+\to\R_+$ be a  smooth monotone increasing function such that 
\begin{itemize}
\item $h$ is $C^2$-small on $[0,1]$, 
\item $h_a'=\frac1{2}\ell_0$ on the interval $[1,2]$, and
\item $h'_a$ increases to $a$ on $[2,\infty)$.
\end{itemize}
See Figure \ref{fig:ha}. We then further introduce another parameter $\rho$ to the function $h_a$. Namely, for $\rho >0$ we define functions $h_{a,\rho}$ by requiring that
\[
h_{a,\rho}'(t):=\begin{cases} h_a'(t),\qquad & t\in[0,1],\\
              \frac1{2}\ell_0,\qquad& t\in[1,2e^\rho],\\
              h_a'(t/e^\rho),\qquad &t\in(2e^\rho,\infty).
       \end{cases}
\]
Like $h_a$, the function $h_{a,\rho}$ has slope $a$ at infinity. The difference is that $h_{a,\rho}$ has slope $\le \ell_0$ on $[0,2e^\rho]$ instead of $[0,2]$. We refer to $\rho$ as the \textbf{delay} parameter. See Figure \ref{fig:ha2}.

\tikzset{every picture/.style={line width=0.75pt}}       
\begin{figure}[h]
  \centering
  \resizebox{10cm}{!}{
\begin{tikzpicture}[x=0.75pt,y=0.75pt,yscale=-1,xscale=1]
\draw  (2,575.74) -- (494,575.74)(50.71,160) -- (50.71,626.06) (487,570.74) -- (494,575.74) -- (487,580.74) (45.71,167) -- (50.71,160) -- (55.71,167)  ;
\draw    (51.2,575.16) .. controls (96.37,564.3) and (137.37,573.3) .. (152.37,568.3) ;
\draw    (250,529.06) -- (152.37,568.3) ;
\draw    (250,529.06) .. controls (359,456.06) and (458,215.06) .. (479,160.06) ;
\draw  [dash pattern={on 0.84pt off 2.51pt}]  (225.26,575.8) -- (456.62,159.01) ;
\draw  [dash pattern={on 0.84pt off 2.51pt}]  (150.7,575.66) -- (485.7,420.66) ;
\draw    (55.75,399.77) -- (45.75,399.77) ;
\draw    (250.75,579.77) -- (250.75,570.77) ;
\draw    (55.75,529.77) -- (45.75,529.77) ;
\draw    (150.75,580.77) -- (150.75,571.77) ;
\draw    (350.75,579.77) -- (350.75,575.06) -- (350.75,570.77) ;
\draw (37,578.4) node [anchor=north west][inner sep=0.75pt]  [font=\scriptsize]  {$0$};
\draw (30,522.4) node [anchor=north west][inner sep=0.75pt]  [font=\scriptsize]  {$\ell_{0}$};
\draw (33,394.4) node [anchor=north west][inner sep=0.75pt]  [font=\scriptsize]  {$a$};
\draw (320,242.4) node [anchor=north west][inner sep=0.75pt]  [font=\footnotesize]  {$a( t-2) +\ell_{0}$};
\draw (425,320.4) node [anchor=north west][inner sep=0.75pt]  [font=\normalsize]  {$h_{a}( t)$};
\draw (433,459.4) node [anchor=north west][inner sep=0.75pt]  [font=\footnotesize]  {$\ell_{0}( t-1)$};
\draw (147,582.4) node [anchor=north west][inner sep=0.75pt]  [font=\scriptsize]  {$1$};
\draw (247,582.4) node [anchor=north west][inner sep=0.75pt]  [font=\scriptsize]  {$2$};
\draw (347,582.4) node [anchor=north west][inner sep=0.75pt]  [font=\scriptsize]  {$3$};
\draw (498,570.4) node [anchor=north west][inner sep=0.75pt]  [font=\small]  {$t$};
\end{tikzpicture}
}
  \caption{The function $h_a(t)$}
  \label{fig:ha}
\end{figure}

\begin{figure}[h]
  \centering
  \resizebox{12cm}{!}{
\tikzset{every picture/.style={line width=0.75pt}} 

\begin{tikzpicture}[x=0.75pt,y=0.75pt,yscale=-1,xscale=1]

\draw  (-54,575.74) -- (1002,575.74)(50.54,160) -- (50.54,626.06) (995,570.74) -- (1002,575.74) -- (995,580.74) (45.54,167) -- (50.54,160) -- (55.54,167)  ;
\draw    (51.2,575.16) .. controls (96.37,564.3) and (137.37,573.3) .. (152.37,568.3) ;
\draw  [dash pattern={on 0.84pt off 2.51pt}]  (248.33,529.67) -- (152.37,568.3) ;
\draw    (250.75,579.77) -- (250.75,570.77) ;
\draw    (55.75,529.77) -- (45.75,529.77) ;
\draw    (150.75,580.77) -- (150.75,571.77) ;
\draw    (650.75,580.77) -- (650.75,571.77) ;
\draw    (649,431) -- (149.9,569.01) ;
\draw  [dash pattern={on 0.84pt off 2.51pt}]  (249,529.06) .. controls (358,456.06) and (457,215.06) .. (478,160.06) ;
\draw    (54.75,430.77) -- (44.75,430.77) ;
\draw    (648.16,431.37) .. controls (782.23,386.86) and (837,231) .. (858,176) ;

\draw (37,578.4) node [anchor=north west][inner sep=0.75pt]  [font=\scriptsize]  {$0$};
\draw (30,522.4) node [anchor=north west][inner sep=0.75pt]  [font=\scriptsize]  {$\ell _{0}$};
\draw (409,370.4) node [anchor=north west][inner sep=0.75pt]  [font=\normalsize]  {$h_{a}( t)$};
\draw (147,582.4) node [anchor=north west][inner sep=0.75pt]  [font=\scriptsize]  {$1$};
\draw (247,582.4) node [anchor=north west][inner sep=0.75pt]  [font=\scriptsize]  {$2$};
\draw (970,584.4) node [anchor=north west][inner sep=0.75pt]  [font=\small]  {$t$};
\draw (642,586.4) node [anchor=north west][inner sep=0.75pt]  [font=\scriptsize]  {$2e^{\rho }$};
\draw (758,375.4) node [anchor=north west][inner sep=0.75pt]  [font=\normalsize]  {$h_{a,\rho }( t)$};
\draw (20,423.4) node [anchor=north west][inner sep=0.75pt]  [font=\scriptsize]  {$\ell _{0} e^{\rho }$};

\end{tikzpicture}}
  \caption{The delayed function $h_{a,\rho}(t)$}
  \label{fig:ha2}
\end{figure}

Let $H^\rho_{a,g}:T^*N\to\R$ be defined by $H^\rho_a(p,q):=h_{a,\rho}(|p|)$ as in \eqref{eq:radial_hamiltonian}. We omit $g$ or $a$ from the notation when they are obvious from the context. In particular, recalling that $|p|=e^r$, we have $H^0_a=h_a(e^r)$. Note that \emph{ for any fixed $\rho$, the set of all Hamiltonians of the form $H^{\rho}_{a,g}$ is $\preceq$-cofinal in the set of all Hamiltonians that are linear at infinity with respect to \emph{any} metric $g'$}.

\begin{rem}
The delay parameter $\rho$  controls the size of the compact subset over which the slope of $H$ is kept small.
\end{rem}

\begin{thm} \label{lmMainEstimate}
Fix a Riemannian metric $g$ on $N$ with non-degenerate time 1 geodesic flow, and a conical almost complex structure $J$ on $T^*N$. For each $a\in\R$ which is not in the period spectrum of the geodesic flow, pick a Hamiltonian $H^{\rho}_a$ depending on the delay parameter $\rho\in\R_+$ as above. Let $\K$ be a compact set of magnetic forms in the $C^{\infty}$ topology. 
Then there is a function $C=C(a,\rho,g,J,\K)$ which is 
\begin{itemize}

\item monotone decreasing in $\rho$, 
\item satisfies $\inf_{\rho}C(a,\rho,J,\K)=1$, 
\item for any pair $\sigma_0,\sigma_1\in\K $, any $\rho\in\R_+$  and any solution $u$ to the parametrised Floer equation
\begin{equation}\label{eqFloerContinuation}
\partial_{s}u+J_{t}(u)(\partial_{t}u-X_{H,\sigma_s}(u,t))=0,
\end{equation}
continuing from $\sigma_0$ to $\sigma_1$, where $\sigma_s=f(s)(\sigma_1-\sigma_0)+\sigma_0$ as in \eqref{eqMagneticFormInterpolation}, we have for each $s\in\R$
\begin{equation}\label{eq_main_estimate}
\left|\int_{(-\infty,s)\times S^1} u^*\pi^*(\sigma_1-\sigma_0)\right|\leq C E(u)\|\pi^*(\sigma_1-\sigma_0)\|_\infty.
\end{equation}
\item
Given a choice of a third magnetic form $\sigma_2\in\K$ we have for any solution to \eqref{eqFloerContinuation} the estimate 
\begin{equation}\label{eq_main_estimate0}
\left|\int u^*\pi^*(\sigma_2-\sigma_0)\right|\leq C E(u)\|\pi^*(\sigma_2-\sigma_0)\|_\infty.
\end{equation}

\end{itemize}
Moreover, the estimates \eqref{eq_main_estimate0} and \eqref{eq_main_estimate} are robust with respect to $C^2$-small perturbations of $H$ and a perturbation of $J$ to one that is metrically equivalent.
\end{thm}

We formulate a variant of the Theorem \ref{lmMainEstimate} for a continuation which involves simultaneously varying $H$ \textit{and} $\sigma$.

\begin{thm}\label{thmEstCont}
Given a  pair $$H_1=H_{a_1,g_1}^{\rho_1}\leq H_2= H_{a_2,g_2}^{\rho_2}$$ and almost conical almost complex structures $J_0,J_1$ as in Theorem \ref{lmMainEstimate}  there is a dissipative Floer datum $\mathfrak{H}^{12}=(H^{12}_s,J^{12}_s)$ interpolating between them and a constant $C^{12}(a_1,a_2,g_1,g_2,J_1,J_2,\rho_1,\rho_2)$ such that  for any $s_0\in\R$ the estimates 
\eqref{eq_main_estimate} and \eqref{eq_main_estimate0} with $C$ replaced by $C^{12}$ hold for solutions to the following Floer equation
\begin{equation}\label{eqFloerContinuation1}
\partial_{s}u+J^{12}_{s+s_0,t}(u)(\partial_{t}u-X_{H^{12}_{s+s_0},\sigma_s}(u,t))=0.
\end{equation}
Moreover, the constant $C^{12}$ is independent of $s_0$ and can be made arbitrarily close to $1$ by taking $\rho_1,\rho_2$ large enough. Given a triple $H_1\leq H_2\leq H_3$ as above and a pair of interpolating dissipative Floer data $\mathfrak{H}^{12}$ and  $\mathfrak{H}^{23}$ as above, the Floer datum $\mathfrak{H}^{123}$ obtained by gluing together $\mathfrak{H}^{12}$ and  $\mathfrak{H}^{23}$ satisfies the same estimates with constant $C^{13}=\max\{C^{12},C^{23}\}$. 

In case where $a_1=a_2=a,g_1=g_2=g,J_1=J_2=J$  we have 
\begin{equation}\label{eqCCtag}
C^{12}(a_1,a_1,g_1,g_1,J_1,J_1,\rho_1,\rho_2)=C(a,g,J,\rho_2).
\end{equation}
Moreover in this case, we can take 
\begin{equation}\label{eqStandIntPl}
H^{12}_{s,t}=f(s)H_1+(1-f(s))H_2
\end{equation}
 for $f$ any function which is monotone decreasing in $s$, equals $1$ for $s\ll 0$, $0$ for $s\gg 0$, and such that $f'(s)$ is compactly supported. Finally, let $H_1= H_{a,g}^{\rho_1}$ and $H_2= H_{a,g}^{\rho_2}+b$ where $b\in\R$ is any constant. Then $C^{12}= C(a,g,J,\min\{\rho_1,\rho_2\})$.  
\end{thm}
\begin{rem}\label{remHomtopoyOfHomotopies}
Taking $s_0$ to the limits $\pm\infty$ in the last Theorem corresponds to concatenating two continuation maps, one fixing $\sigma$, the other fixing $H$, the order determined by whether we are taking the limit to negative or positive infinity.
\end{rem}
At last, we formulate a variant involving the pair of pants. We consider Floer data $(H^\rho_{a_i},J_i)$ for $i=1,2,3$ where $a_i$ are integers and $a_3=a_1+a_2$, and we fix a product Floer datum $(\mathfrak{H},J)$ interpolating between them on $\Sigma_{0,1,2}$ which is everywhere of the form $(H\alpha,J)$ for an appropriate Hamiltonian $H$ and closed form $\alpha$. We let $\eta \colon \Sigma_{0,1,2}\to\R$ be a proper function with negative ends  mapping to $-\infty$ and the positive end to $+\infty$. Let $\eta_{s_0}(s):=\eta(s+s_0)$.  For an $s$-dependent magnetic form $\sigma_s$ we obtain a family 
\begin{equation}\label{eqPoPfam}
z\mapsto \sigma_{\eta_{s_0}(z)}
\end{equation}
on $\Sigma_{0,1,2}$.  Note that moving $s_0$ to $-\infty$ has the effect of first applying the magnetic continuation and then the product, while moving $s_0$ to $\infty$ has the effect of doing the same in the reverse order.
\begin{thm}\label{thmEstProd}
There exists a product Floer datum $\mathfrak{H}$ such that the statement of Theorems \ref{tmEneCompCont} and \ref{lmMainEstimate} remain true upon replacing the cylinder by the pair of pants, replacing $H$ by $\mathfrak{H}$ and taking $\sigma$ to be the family given in \eqref{eqPoPfam}. The resulting constant $C$ can be taken to be independent of $s_0$.
\end{thm}

\subsection{Proof of Theorem~\ref{tmMainIso}}
\label{sec:proof_main}
We now prove our main result, using the (as yet unproved) results from the previous section.
\begin{proof}[Proof of Theorem~\ref{tmMainIso}]
As our set $\mathcal{H}_K$ we consider pairs constructed as follows: By Lemma \ref{lmRescale} we can choose $J$ so that 
\begin{equation}\label{eqSigmainfyEst}
\|\pi^*\sigma_k\|_\infty<1/4,\quad\forall \, k\in\K.
\end{equation}
Then for any $a$ not in the associated period spectrum let $H$ be sufficiently $C^2$-close to $H^\rho_{a,g}$ with $\rho$ such that 
\begin{equation}\label{eqCLess2}
C(a,\rho,J,g,\sigma(\K))<2.
\end{equation}

By construction $\mathcal{H}_\K$ is open and $\preceq$-cofinal in the set of all linear at infinity Hamiltonians. Clearly, $\mathcal{H}_\K\subset\mathcal{H}_{\K'}$. 
This establishes the properties \ref{tmMainIsoA}-\ref{tmMainIsoC} of Theorem \ref{tmMainIso}.

We now establish  Theorem  \ref{tmMainIso}\ref{tmMainIsoD}.
We have that $CF(H,\K',k)$ is finitely generated over $\Lambda_{\K'}$ by critical points  of $\mathcal{A}_{H,\sigma_k}$. Let $n$ be the number of $1$-periodic orbits of the Hamiltonian flow of $H$ \emph{with respect to $\omega_{\sigma_k}$}. Fix a basis $$x_i=[\gamma_i,w_i],\quad 1\leq i\leq n$$ of such critical points. A solution $u$  to Floer's equation
\begin{equation}\label{eqFloerSolutionOk}
\partial_{s}u+J(u)(\partial_{t}u-X_{H,\omega_{\sigma_k}}(u,t))=0,
\end{equation}
 determined by $\omega_{\sigma_k}$ gives rise to a $\Lambda_{\K'}$ linear map
 $$
 d_u:CF(H,\K',k)\to CF(H,\K',k).
 $$
 Explicitly, in the basis chosen above, if $u$ connects the periodic orbits $\gamma_i,\gamma_j$ corresponding to the basis elements $x_i,x_j$ respectively then $d_u$ is a matrix whose entries are 
 $$
 d_{u,l,m}=\begin{cases}e^{w_i\#u\#-w_j},&\quad (l,m)=(i,j),\\
           0,&\quad (l,m)\neq (i,j).
          \end{cases}					
$$
A priori $d_{H,\omega_k}$ is the  formal sum $d=\sum d_u$ where $u$ ranges over all Floer trajectories (modulo the action of $\R$ by shifts) of index difference $1$. We need to show this sum converges. In other words, we need to show that for each $(i,j)$ the set of elements  $\{w_i\#u\#-w_j\}\subset \Gamma$, where $u$ ranges over all the Floer trajectories (modulo the action of $\R$ by shifts)  of index difference $1$ connecting $\gamma_i$ and $\gamma_j$, satisfies the multi-finiteness conditions determined by $\K'$. By definition this means the following.  For any real number $c$ and any $k'\in \K'$ let $n_{c,k'}$ be the number of such Floer trajectories $u$ so that $\omega_{k'}(u):=\langle\omega_{k'},w_i\#u\#-w_j\rangle<c$. Here $\langle\cdot,\cdot\rangle$ is the pairing between singular homology and de Rham cohomology on $T^*N$. Then we need to show that $n_{c,k'}$ is finite for all $c\in\R$ and $k'\in\K$. 

By equation \eqref{eqActionEnergyIdentity} for the case $\sigma_0=\sigma_1$, we have 
\begin{align}
\omega_{k'}(u)= & \ \omega_k(u)+\int u^*\pi^*\sigma+\int w_i^*\pi^*\sigma-\int w_j^*\pi^*\sigma\notag\\
= & \ E(u)-\int (H(t\gamma_i(t))-H(t,\gamma_j(t))dt \notag\\
 & +\int u^*\pi^*\sigma+\int w_i^*\pi^*\sigma-\int w_j^*\pi^*\sigma\notag\\
= & \ E(u)+\int u^*\pi^*\sigma+D_{ij}\notag,
\end{align}
where $D_{ij}$ is a $u$-independent constant. Applying the estimates \eqref{eq_main_estimate0}, \eqref{eqSigmainfyEst} and \eqref{eqCLess2}, we have $$\left|\int u^*\pi^*\sigma\right|<\frac12E(u).$$ So $n_{c,k'}$ is estimated by the number $n'$ of solutions $u$ to \eqref{eqFloerSolutionOk} of energy $\leq 2(c'+ D_{ij})$. These solutions are all contained in an a priori compact set by Theorem \ref{tmDissDaimEst} and Theorem \ref{tmLinDiss}. So for $(H,J)$ regular, the number $n'$ is finite by standard Gromov-Floer compactness.

We proceed to prove Theorem \ref{tmMainIso}\ref{tmMainIsoE}. For $k_1,k_2\in\K$ and $\K'\subset\K$ we  wish to construct the continuation map
\[
f:CF(H,\K',k_1)\to CF(H,\K',k_2).
\]
Let $\{\alpha_1,\dots,\alpha_m\}$ and $\{\beta_1,\dots,\beta_n\}$  be the $1$-periodic orbits of $H$ with respect to $\omega_{k_1},\omega_{k_2}$ respectively. Pick lifts $\{[\alpha_i,w_i]\}$ and $\{[\beta_i,v_i]\}$ to the appropriate covering space. These form ordered bases for $CF(H,\K',k_1)$ and $CF(H,\K',k_2)$ respectively over $\Lambda_{\K'}$. 

Let $u$ be a solution to Floer's equation \eqref{eqFloerContinuation}  
interpolating between $\sigma_{k_1}$ and $\sigma_{k_2}$ via \eqref{eqMagneticFormInterpolation}. We have
\[
\mathcal{M}([\alpha_i,w_i],[\beta_j,v_j];H,\{\sigma_s\})=\mathcal{M}([\alpha_i,w_i\#w],[\beta_j,v_j\#w];H,\{\sigma_s\}),
\]
for all loops $w$ in the relevant component of the loop-space. Thus, as before, a continuation Floer trajectory $u$ from $[\alpha_i,w_i]$ to $[\beta_j,v_j]$ defines a $\Lambda_{\K'}$-linear map 
\[
f_u:CF(H,\K',k_1)\to CF(H,\K',k_2)
\] 
given in the bases above by the matrix 
 $$
 f_{u,l,m}=\begin{cases}e^{w_i\#u\#-v_j},&\quad (l,m)=(i,j),\\
           0,&\quad (l,m)\neq (i,j).
          \end{cases}					
$$

The map $f$ is formally the sum $f=\sum_u f_u$, where the sum is over all solutions $u$ to \eqref{eqFloerContinuation}. We need to verify that this sum converges. Writing
\[
\Delta\mathcal{A}_u:=\mathcal{A}_{H,\sigma_0}([\alpha_j,w_i])-\mathcal{A}_{H,\sigma_1}([\beta_j,v_j\#u]),
\]
we obtain by \eqref{eqCLess2} and \eqref{eq_main_estimate} and \eqref{eqActnEst} that
\[
E(u)\leq\Delta\mathcal{A}_u+|e|\leq \Delta\mathcal{A}_u+ \left|\int w^{0*}\sigma\right|+\frac12 E(u).
\]
This translates into 
\[
E(u)\leq 2\left(\Delta\mathcal{A}_u+\left|\int w^{0*}\sigma\right|\right).
\]
As a consequence we have an energy estimate for Floer trajectories of each topological type and we deduce by Gromov-Floer compactness that that \emph{the number $n_{A,i,j}$ of solutions \eqref{eqFloerContinuation} connecting a given pair of periodic orbits  indexed by $i,j$ and of topological type $A$ is finite.} It remains to verify the multi-finiteness conditions. For this it suffices to show that for any $k'\in \K'$ an estimate on $\omega_{k'}(u)$ implies one on $E(u)$. 
Taking $D:=\max_i \left\{\left|\int w^{i*}\sigma\right|\right\}$  we can rewrite the previous estimate to obtain
\[
E(u)\leq 2(\Delta\mathcal{A}_u+D)\leq 2\omega_{\sigma_0}(u)+F
\]
for all possible continuation trajectories. Here $F$ is a constant estimating $D$, the Hamiltonian term appearing in the formula for $\mathcal{A}_{H,\omega_{\sigma_i}}$   and the value of $\omega_{\sigma_1}$ on $v_j$. All these terms depend on $i,j$ but not on $u$. 

Given any $k'\in\K$ we have the estimate 
\[
\omega_{\sigma_0}(u)\leq \omega_{k'}(u)+\frac{C}{4}E(u)
\]
Combining the last two estimates with \eqref{eqCLess2} we get the desired estimate on $E(u)$ in terms of $\omega_k'(u)$. 

We have thus established the well-definedness of $f$. We define 
\[
f':CF(H,\K,k_2)\to CF(H,\K,k_1).
\]
in the same way. Since Gromov-Floer compactness holds in our setting, the argument that $f$ and $f'$ are chain maps is the standard one. Similarly, we construct chain homotopies between $f\circ f'$, $f'\circ f$ and the respective identities in the usual way with no new ideas required beyond what was needed in the definition of $f,f'$.

Finally we deal with Theorem \ref{tmMainIso}\ref{tmMainIsoF}. We first deal with  the case $(H_{a_1,g_1}^{\rho_1},J_2)\leq (H_{a_2,g_2}^{\rho_2},J_2)$ where  $\rho_1$ and $\rho_2$ are small enough for the constant $C^{12}$ of Theorem \ref{thmEstCont}  to be $<2$.  In this case, we pick a continuation datum as in Theorem \ref{thmEstCont}. Each solution $u$ to \eqref{eqFloerContinuation} then gives rise to $\Lambda_{\cK}$ linear map 
$$
c_u:CF(H_1,\K,k_2)\to CF(H_2,\K,k_1).
$$
Given the estimates of Theorem \ref{thmEstCont}, the convergence of the sum $\Sigma_u c_u$ where $u$ runs over all solutions over all $u$ to \eqref{eqFloerContinuation} is word for word the same as for the differential. The fact that this gives rise to a map on homology, i.e., it commutes with the differential, is standard. To show that $f_{k_1,k_2,*}$ commutes with these continuation maps  we consider the homotopy of homotopies  as described in Remark \ref{remHomtopoyOfHomotopies} where we allow equation \eqref{eqFloerContinuation} to depend on additional parameter $s_0$ and obtain a homotopy between the data defining $h\circ f$ and the data defining $f\circ h$ by allowing $s_0$ to go between from $-\infty$ to $+\infty$. Theorem \ref{thmEstCont} and the same argument as in the construction of $f$ shows that this homotopy produces a chain homotopy between $h\circ f$ and $f\circ h$. 

To define continuation maps for \emph{all} pairs $H_1\preceq H_2$ in $\cH_{\K}$ note that by the argument in the last paragraph and \eqref{eqCCtag},  the continuation map is defined for all ordered pairs of the form $H_1=H_{a,g}^{\rho_1}\leq H_2=H_{a,g}^{\rho_2}$ whenever both Hamiltonians satisfy \eqref{eqCLess2}. Moreover, the last paragraph of Theorem \ref{thmEstCont} allows us to show that in this case the continuation map is an isomorphism.  Indeed, for any $b>0$, if we consider the continuation from $H_1=H_{a,g}^{\rho_1}$ to $H_3=H_1+b$ given by $H^{13}_s=H_1+f(s)$, the continuation map is the identity. If we pick $b>\max\{H_2-H_1\}$, we get a pair of continuation data $H^{12},H^{23}$  gluing gives us another continuation datum, $H^{123}$ from $H_1$ to $H_3$. If we pick $H^{12}$ and $H^{13}$ within the set of continuation data of the form \eqref{eqStandIntPl} there is homotopy from $H^{123}$ to $H^{13}$ within the set of continuation data of the form \eqref{eqStandIntPl} as standard argument now implies the maps on homology induced by $H^{13}$ and $H^{123}$ are the same. 

This allows us to define a continuation map for an arbitrary pair $H_{a_1,g_1}^{\rho_1}\preceq H_{a_2,g_2}^{\rho_2}$, by applying for each of them a continuation map to one of the same slope but with sufficiently large delay parameter for $C'$ to be $<2$ and the previous paragraph to apply \footnote{Note we are not 
proving that this continuation maps is independent of the choices.}.  Commutation of $f_{k_1,k_2,*}$ with  continuation maps for all $\preceq$-ordered pairs of Floer data in $\cH_{\K}$ is now automatic.

The claim concerning the product follows in the same way from Theorem \ref{thmEstProd}. 
\end{proof}
\section{Proof of Theorem \ref{lmMainEstimate}}
\label{sec:estimate_proofs}

The proof of Theorem \ref{lmMainEstimate} will rely on the following two technical lemmas whose proof we postpone until later in the section. For each $\rho$ subdivide $T^*N$ into three regions, $B_i^{\rho}$ for $i=1,2,3$, corresponding respectively to $e^r\in[0,1)$, $e^r\in[1,2e^\rho]$ and $e^r\in[2e^\rho,\infty)$. Denote by $\psi_\rho$ the Liouville flow. Then $H^{\rho}_a$ coincides with $H^0_a\circ\psi_{-\rho}$ on $B_3^{\rho}$ and with $H^0_a$ on the interval $B_1^{\rho}$. On the intermediate region $B_2^\rho$, after perhaps shrinking $\ell_0$, there are no periodic orbits magnetic or otherwise.

\begin{lem}\label{lm_est1}
Fix an $a\in\R_+$ not in the period spectrum of the geodesic flow. For every $\delta>0$ there is an $\epsilon>0$ such that for any $\rho\geq 0$ and any $x:S^1\to T^*N$ satisfying $E_{H^{\rho}_a,J,\omega_0}(x)<\epsilon$ there is an $\omega_0$-periodic orbit $x_0$ of $H^{\rho}_a$ for which
\[
\sup_td(x(t),x_0(t))< \delta.
\]
\end{lem}

We shall need the following variation which is a corollary.
\begin{cor}\label{cor_est1}
For any compact set $\K$ in the space of magnetic forms with the $C^{\infty}$ topology and for any $\delta>0$ there are constants $\epsilon>0$ and $\rho_0$ such that the following holds:  Let  $\sigma\in\K$ and suppose $\rho\geq \rho_0$ and let $x \colon S^1\to T^*N$ satisfy $E_{H^{\rho},J,\omega_\sigma}(x)<\epsilon$ then there is an $\omega_0$-periodic orbit $x_0$ of $H^\rho$ for which
\[
\sup_td(x(t),x_0(t))< \delta.
\]

\end{cor}
\begin{proof}
The constant $\ell_0$ can be chosen so that on the intermediate region $B_2^{\rho}$ there is a lower bound on the $\omega_\sigma$-energy of any loop. Thus, for $\epsilon$ small enough, the assumption $E_{H^{\rho},J,\omega_\sigma}(x)<\epsilon$ implies that $x$ is contained either in $B_1^\rho$ or in $B_3^\rho$. If $x$ is contained in $B^{\rho}_1$, the inequality further implies that the gradient of $H^{\rho}$ with respect to $\omega_\sigma$ is arbitrarily small depending on $\epsilon$, and thus $x$ is arbitrarily close to a critical point of $H^{\rho}$. For $x$ in $B_3^\rho$ and $\rho$ large enough, the inequality   $E_{H,J,\omega_\sigma}(x)<\epsilon$ implies the inequality $E_{H,J,\omega_0}(x)<2\epsilon$. To see this, by Lemma \ref{lmMagPertbDec} we have that the magnetic correction to the flow of $|p|=e^r$ has decaying norm with respect to conical metrics. The claim thus follows from Lemma~\ref{lm_est1}.
\end{proof}
\begin{lem}\label{lm_est2}
For any compact set $\K$, any pair $\sigma_0,\sigma_1\in \K,$ and any $\epsilon>0$, there are constants $\delta>0$ and $b$ such that the following holds: For any $\rho>0$ and any solution $u$ to the parametrised Floer equation
\begin{equation}
\partial_{s}u+J_{s,t}(u)(\partial_{t}u-X_{H,\sigma_s}(u,t))=0,
\end{equation}
where $\sigma_s=f(s)(\sigma_1-\sigma_0)+\sigma_0$ as in \eqref{eqMagneticFormInterpolation} defined on the cylinder $[s_0-b,s_0+b]\times S^1$ for some $s_0\in\R$ we have
\[
E_{H^{\rho},J_{s_0},\sigma_{s_0}}(u(s_0))\geq\epsilon\quad\Rightarrow\quad E_{H^{\rho},J}(u;[s_0-b,s_0+b]\times S^1)\geq\delta.
\]
\end{lem}

\begin{proof}[Proof of Theorem~\ref{lmMainEstimate}]
Write $H:=H^{\rho}$. In what follows, observe that near infinity $\|X_H\|$ scales like the distance while $\pi^*\sigma$ scales like the inverse of the distance squared. Thus, expressions like $|\iota_{X_H}\pi^*\sigma|_{\infty}$ or $\|X_{H}(p)\||\pi^*\sigma_p|$ are bounded uniformly and in fact decay as the inverse of the distance. If we fix a compact set $\K$ of magnetic forms $\sigma$, this decay is uniform in $\K$ as well.

Let $\sigma_0,\sigma_1,\sigma_2\in\K$. Let $u$ be a solution to the parametrised Floer equation
\begin{equation}
\partial_{s}u+J_{s,t}(u)(\partial_{t}u-X_{H,\sigma_s}(u,t))=0,
\end{equation}
where $\sigma_s=f(s)(\sigma_1-\sigma_0)+\sigma_0$ as in \eqref{eqMagneticFormInterpolation}.

For $\epsilon>0$, let $A_{\epsilon}\subset\R$ consist of values $s$ such that $E(u(s,\cdot))\geq\epsilon$. Let $B_\epsilon\subset \R$ consist of values whose distance from $A_\epsilon$ is at most $b(\epsilon)$, where $b$ is the constant determined by $\epsilon$ according to Lemma \ref{lm_est2}. From Lemma \ref{lm_est2} we have for the Lebesgue measure $\mu$ that
\[
\mu(B_{\epsilon})\leq \frac{2b}{\delta}E_{H,J}(u).
\]
In the following let $\sigma$ denote either $\sigma_1-\sigma_0$ or $\sigma_2-\sigma_0$. One uses Floer's equation to establish for any $B\subset B_{\epsilon}$ the inequality
\begin{align}\label{estC1}
\left|\int_{B\times S^1}u^*\pi^*\sigma\right|&\leq\int_{B\times S^1}\|\partial_su\|^2\|\pi^*\sigma\|_{\infty}+\left|\int_{B\times S^1}{ \sigma}(\partial_su,X_H)\right|\notag\\
&\leq \int_{B\times S^1}\|\partial_su\|^2\|\pi^*\sigma\|_{\infty}+\int_{B\times S^1}\|\iota_{X_H}\pi^*\sigma\|_{\infty}\max\{1,\|\partial_su\|^2\}\notag\\
&\leq \left(1+\left(\frac{2b}{\delta}+1\right)\frac{\sup|\iota_{X_H}\pi^*\sigma|_{\infty}}{|\pi^*\sigma|_{\infty}}\right)E_{H,J}(u)\|\pi^*\sigma\|_{\infty}\notag\\
&=:C_1(a,\rho,\K){E_{H,J}}(u)\|\pi^*\sigma\|_{\infty}.
\end{align}
Here the dependence on $a$ enters through the dependence on $H$ which has slope $a$ at infinity.

Let $D:=\R\setminus B_{\epsilon}$. We need to control $\int_{D\times S^1}u^*\sigma$. The number of components of $D$ is bounded by $E_{H,J}(u)/\delta$. Let $D_0$ be a component of $D$.  Fix a constant $\delta_1$. For any $s\in D_0$ we have that $E(u(s))<\epsilon$. Thus, by Corollary~\ref{cor_est1} and by making $\epsilon$ small enough, we may assume that, for $\rho$ large enough,  $u$ maps $D_0$ into the $\delta_1$-tubular neighbourhood of some periodic orbit $\gamma$ of $H^{\rho}$ with respect to $\omega = \omega_0$. We can take this neighbourhood to be the pre-image under $\pi$ of a tubular neighbourhood of $\pi(\gamma)$ where $\sigma$ is exact with a primitive $\alpha$. Note that projections of periodic orbits depends on $a$ but not on $\rho$. Thus there is a uniform estimate on $\pi^*\alpha$ which is independent of $\rho$. Consider the function $f=\iota_{X_H}\pi^*\alpha$. Then $|df|$ is scaling invariant outside a compact set and so decays like the inverse of the distance. Denote the components of $\partial D_0$ by $\gamma_0,\gamma_1$. By Stokes Theorem we have
\begin{align}
\left|\int_{D_0}u^*\pi^*\sigma\right|&=\left|\int_{\gamma_1-\gamma_0}u^*\pi^*\alpha\right|\notag\\
&\leq \int_{S^1}|f\circ\gamma_1-f\circ\gamma_0|dt+|\pi^*\alpha|(E(\gamma_1)+E(\gamma_0))\notag\\
&\leq|df|\delta_1+2\epsilon\sup_{\gamma_1\cup\gamma_0}|\pi^*\alpha|.\notag
\end{align}
Thus
\begin{align}\label{estC2}
\left|\int_Du^*\omega_\sigma \right| & \leq\frac1{\delta\|\pi^*\sigma\|_{\infty}}(|df|\delta_1+2\epsilon\sup_{\gamma_1\cup\gamma_0}|\pi^*\alpha|)E(u)\|\pi^*\sigma\|_{\infty}\notag \\
& =:C_2(a,\rho,\K)E(u)\|\pi^*\sigma\|_{\infty}.	
\end{align}
Moreover, the estimate remains true if, for any $s\in\R$, we replace $D$ on the left hand side of \eqref{estC2} with $D\cap (-\infty,s)\times S^1$. The constant on the right hand side is then independent of $s$.

We define the constant $C(a,\rho,K)=C_1(a,\rho,\K)+C_2(a,\rho,\K)$. We now verify that $C(a,\rho,\K)$ satisfies all the listed properties. By the observations concerning the scaling behaviours of $X_H$, $\sigma$ and $\alpha$, we have that  $C(a,\rho,K)$ is monotone increasing in $\rho$ and that $C_1(\rho)\to 1$ and $C_2(\rho)\to 0$ as $\rho\to\infty$. The estimates \eqref{eq_main_estimate} and \eqref{eq_main_estimate0} are immediate consequences of the estimates \eqref{estC1} and \eqref{estC2}.

\end{proof}

\begin{rem}  The proof of Theorem~\ref{lmMainEstimate} is robust under sufficiently small (compactly supported) time-dependent perturbations. In fact, on the unit disc bundle it suffices to require that $H^\rho$ be equal to some fixed $C^2$-small Hamiltonian. On the other hand, even in the compact case, when varying the symplectic structure one can only obtain the required a priori energy estimates for  Hamiltonians which are small enough relative to the perturbation of the symplectic form. Thus, it is to be expected that besides linearity near infinity one must further restrict the behaviour of the Hamiltonian in a neighbourhood of the zero section. This explains why we were forced to introduce the delay parameter $\rho$ to our Hamiltonians.
\end{rem}
\begin{proof}[Proof of Theorems \ref{thmEstCont} and \ref{thmEstProd}]
Consider first the case of Theorem \ref{thmEstCont}. Let $a_1\leq a_2$, let $f:\R^2\to[0,1]$ be a zig-zag function as in the proof of Theorem \ref{tmContinuatiation} and let $f_{s_0}(u,v):=f(u+s_0,v)$. Let $\eta_{s_0}:\R\times T^*M\to[0,1]$ be defined by $\eta_{s_0}(s,p)=f_{s_0}(s,\sqrt{|p|})$ and let $H_{s_0}$ be an $s$-dependent Hamiltonian defined by
\[
H_{s_0,s}:=\eta_{s_0}(s)H^\rho_{a_1}+(1-\eta_{s_0}(s))H^\rho_{a_2}.
\]
Call the region $\{(s,t)|s\in[s_0,s_0+1]\}$ where $\eta_{s_0}$ is non-constant the \textbf{zig-zag region}. The only difference with the case of Theorem \ref{lmMainEstimate} is that Lemma \ref{lm_est2} does not hold on the zig-zag region as the equation is not of the form \eqref{eqFloerContinuation} there. Thus, by the proof of Theorem \ref{lmMainEstimate} we have for any subset $F\subset \R\setminus[s_0-1,s_0+2])\times S^1$ the estimate 
\begin{equation}
\int_{F}u^*\pi^*\sigma\leq C'' E(u),
\end{equation}
where $C''=\max\{C(a_1,\rho_1,g_1), C(a_2,\rho_2,g_2)\}$. 
To estimate the integral $\int_{[s_0-1,s_0+2]\times S^1}u^*\pi^*\sigma$.  For this, Let $F\subset[s_0-1,s_0+2]\times S^1$, and proceed exactly as in the proof of the estimate \eqref{estC1}. We obtain
\begin{equation}
\left|\int_{F}u^*\pi^*\sigma\right|
\leq \int_{F\times S^1}\|\partial_su\|^2\|\pi^*\sigma\|_{\infty}+\int_{F\times S^1}\|\iota_{X_{H_{s_0}}}\pi^*\sigma\|_{\infty}\max\{1,\|\partial_su\|^2\}\notag.
\end{equation}

To conclude  we need to show  $\|\iota_{X_{H_{s_0}}}\pi^*\sigma\|_{\infty}$  can be made arbitrarily small. 
Computing $X_{H_{s_0,s}}$ we see that 
$$\|X_{H_{s_0,s}}\|\leq2\max\left\{\eta_{s_0}\|X_{H^\rho_{a_2,g_2}}\|,H^\rho_{a_2,g_2}\|X_{\eta_{s_0}}\|\right\}.$$ The term $\|X_{\eta_{s_0}}\|$ can be made arbitrarily small by making $\partial_vf$ arbitrarily small. Since $H^\rho_{a_2,g_2}$ scales as $e^r$, while the 2-form $|\pi^*\sigma|$ scales as $e^{-r}$, the we obtain that $\|\iota_{X_{H_{s_0}}}\pi^*\sigma\|_{\infty}$ is arbitrarily small as $\rho\to\infty$. Thus as in \eqref{estC1} we conclude,
$$\left|\int_{F}u^*\pi^*\sigma\right|<(1+C_3(\rho))E(u)\|\pi^*\sigma\|_{\infty}$$ for $C_3$ a constant converging to $0$ as $\rho=\min\{\rho_1,\rho_2\}\to\infty$.  

Finally, to see that \eqref{eqCCtag} holds, note that for a continuation from $H_{a,g}^{\rho_1}\leq H_{a,g}^{\rho_2}$ with $\rho_2\leq \rho_1$, the continuation map is of the form \eqref{eqFloerContinuation} in the region $B_3^{\rho}$. So the estimates of Theorem \ref{lmMainEstimate} hold with the constant $C(\rho_1)$.

We now treat Theorem \ref{thmEstProd}. In this case we are considering a product datum of the form $H\alpha$. So there is no zig-zag region. However, we need to briefly comment on what happens in the complement of the cylindrical ends, where Floer's equation is of the form \eqref{eqProdFloer} rather than \eqref{eqFloer}. Without loss of generality we can take $\alpha$ to be the pullback of the form $dt$ on the cylinder $\R\times S^1$ by the branched covering  $\psi: z\mapsto \Log z^n(z-1)^m$ where we identify the surface $\Sigma_{0,1,2}$ with the punctured Riemann sphere $S^2\setminus\{0,1,\infty\}$. Then away from the branch points which have measure 0, $\psi$ induced coordinates on $\Sigma_{0,1,2}$ in which \eqref{eqProdFloer} is exactly \eqref{eqFloer}. So the derivation of the estimate goes through for the more general equation.

\end{proof}

\subsection{Proof of Lemmas \ref{lm_est1} and \ref{lm_est2}}
The proof of Lemma \ref{lm_est1} relies on the following general lemma about Hamiltonian flows.

\begin{lem}
Let $(M,\omega)$ be a symplectic manifold, let $H \colon M\to\R$ and let $x_0 \colon S^1\to M$ be a transversally non-degenerate periodic orbit of the Hamiltonian flow of $H$. There are constants $c=c(x_0,H)$ and $\epsilon=\epsilon(x_0,H)$ such that for any $x:S^1\to M$ satisfying $d(x,x_0)<\epsilon$ we have
\begin{equation}\label{eqDistEnESt}
 \sup_td^2(x(t),x_0(t))<c^2\int_0^1\|\dot{x}(t)-X_H(x(t))\|^2dt.
\end{equation}
\end{lem}
\begin{proof}
Let $\epsilon$ be the radius of injectivity of the normal exponential to $x_0$. Suppose \eqref{eqDistEnESt} does not hold. By Cauchy-Schwartz there is a sequence $x_n:S^1\to B_\epsilon(x_0)$ converging in $C^0$ to $x_0$ such that $x_n\neq x_0$ and such that
\[
\lim_n\frac{\int_0^1\|\dot{x_n}(t)-X_H(x_n(t))\|dt}{d(x_n,x_0)}=0.
\]

Abbreviate $d_n=\sup_td(x_0(t),x_n(t))$ and write in exponential coordinates $v_n(t):=\frac{x_n(t)-x_0(t)}{d_n}.$ We will show that we have
\begin{equation}\label{eq2}
\lim_{n\to\infty}\int_0^1\|\dot{v}_n(t)-dX_H(x_0(t))v_n(t)\|dt=0.
\end{equation}
By standard ODE theory this implies that $v_n$ converges to a 1-periodic solution $v$ to the differential equation
\[
\dot{v}(t)=dX_H(x_0(t))v(t).
\]
Moreover, the vertical component of $v(t)$ does not vanish identically.

Denote by $\psi_t$ the flow of $X_H$. Let $v(t_0)$ have non-vanishing vertical component and let $v_0:=d\psi_{t_0}^{-1}v(t_0)$. Then
\begin{equation}\label{eq1}
v(t)=d\psi_tv_0,
\end{equation}
since both satisfy the same differential equation with the same initial condition. In particular, $d\psi_1v_0=v_0$ contradicting the non-degeneracy of $x_0$.

It remains to prove Equation~\eqref{eq2}. We have
\begin{align}
\dot{v}_n&=\frac{\dot{x}_n-\dot{x_0}}{d_n}\notag\\
&=\frac{\dot{x}_n-X_H\circ x_n}{d_n}+\frac{X_H\circ x_n-X_H\circ x_0}{d_n}\notag\\
&=\frac{\dot{x}_n-X_H\circ x_n}{d_n}+dX_H\circ x_0(v_n)+\frac1{d_n}O(d_n^2).\notag
\end{align}
 Thus
\[
\int_0^1\|\dot{v}_n(t)-dX_H( x_0(t))(v_n(t))\|dt\leq \int_0^1\frac{\|\dot{x}_n-x_H\circ x_n\|}{d_n}+\frac1{d_n}O(d_n^2).
\]
This goes to $0$ by the assumption, completing the proof.
\end{proof}
\begin{proof}[Proof of Lemma \ref{lm_est1}]
Since $a$ is not in the period spectrum and the metric is conical we have that any loop with $E_{H^0}$-energy less than a given constant is contained a priori in some compact set. Thus the claim of the lemma holds for $\rho=0$ by a standard Arzela-Ascoli argument. It remains prove that for a given $\delta$, an $\epsilon$ can be chosen such that the claim will hold for  all $\rho\geq0$. Let $x$ satisfy $E_{H^\rho}(x)\leq \epsilon$. For $\epsilon$ small enough, independently of $\rho$, this implies $x$ is contained either in $B_1^\rho$ or $B_3^\rho$. In the first case, since $H^0=H^\rho$ on $B_1^{\rho}$, the claim is clear. In the second case, $x$ is conjugate via scaling by $e^{-\rho}$ to the loop $x_0:=e^{-\rho}x$ such that $E_{H^0}(x_0)<e^{-\rho}\epsilon$. See \S\ref{SeqACS} for the scaling behaviour of the norm under multiplication by $e^\rho$. So, fixing $\delta$ and taking $\epsilon$ small enough, independently of $\rho$, there is an $\omega_0$-periodic geodesic $x_1$ of $H^0$ satisfying $d(x_1,x_0)<\delta$. If $\delta$ is small enough then according to \eqref{eqDistEnESt} we have $d^2(x_0,x_1)<c^2E_{H^0}(x^0)=e^{-\rho}E_{H^\rho}(x)$. Scaling back up, we get the periodic orbit $x_2:=e^\rho x_1$ of $H^\rho$ and the estimate $d^2(x,x_2)<c^2E_{H^\rho}(x_1)$, which implies the required estimate.
\end{proof}
\begin{proof}[Proof of Lemma \ref{lm_est2}]
We rely on the following property of Floer trajectories which is stated in \cite{Salamon1999} for the case of closed symplectic manifolds. There are constants $\hbar$ and $c$  such that
\begin{equation}\label{mean_value_ineq}
\int_{B_{r}(s,t)}|\partial_su|^2<\hbar\quad\Rightarrow\quad|\partial_su|^2(s,t)<\frac8{\pi r^2}\int_{B_{r}(s,t)}|\partial_su|^2+cr^2.
\end{equation}

Strictly speaking, to apply this to our setting we need to first verify a few things. First we need to show that such constants exist in our non-compact setting. For this note that the estimate~\eqref{mean_value_ineq} can be deduced from the corresponding mean value inequality for $J$-holomorphic curves established in \cite{McDuffSalamon2004}. Namely, if $v$ is a $J$-holomorphic curve in a geometrically bounded manifold we have
\[
\int_{B_{r}(s,t)}|\partial_sv|^2<\hbar\quad\Rightarrow\quad|\partial_sv|^2(s,t)<\frac8{\pi r^2}\int_{B_{r}(s,t)}|\partial_sv|^2
\]
for appropriate constants. The estimate \eqref{mean_value_ineq} can be deduced from this by the Gromov trick as in Definition \ref{dfJH} and the discussion right after. The constants in the mean value inequality of \cite{McDuffSalamon2004} depend on the sectional curvature and the $C^2$-norm of $J$, all with respect to the metric determined by $\omega$ and $J$. We thus need to establish the corresponding bounds on the geometry of the metric determined by the Gromov trick. For this note first that in the case of no magnetic form, the sectional curvature of the Gromov metric for a Hamiltonian which is linear at infinity paired with a conical almost complex structure is shown to be bounded in   \cite{Groman2015}. As in first half of the proof of Theorem \ref{tmLinDiss} the magnetic term decays as $r\to\infty$. Moreover this holds true also for all the derivatives of the magnetic term by vector fields of unit norm. It follows that the sectional curvature remains bounded in the presence of the magnetic term. The same reasoning applies to the $C^2$-norm of $J_{X_H}$. The latter is estimated by the $C^2$ norm of $J$ together with the $C^3$ norm of $X_H$.

We also need to verify that for fixed $H^0$, the constants $\hbar$ and $c$ are independent of $H^\rho$. For this note that in \cite{Groman2015} it is shown that the sectional curvature of the Gromov metric is controlled by that of the underlying manifold and the $C^2$-norm of $H$. The latter decreases under rescaling.

Having established the mean value inequality \eqref{mean_value_ineq} with uniform constants, we prove the lemma. In a form more suitable to our applications we may write
\begin{gather}\label{mean_value_ineq_1}
E_{H^{\rho},J}(u;[s-r,s+r]\times S^1)<\hbar\quad\Rightarrow\\\quad E_{H^{\rho},J_s,\sigma_s}(u_s)<\frac8{\pi r^2}E_{H^{\rho},J}(u;[s-r,s+r]\times S^1)+cr^2,\notag
\end{gather}
for $r\leq 1$. Given $\epsilon>0$ as in the lemma, let $b=\min\{1,\sqrt{\frac{\epsilon}{2c}}\}$. Without loss of generality $E_{H^{\rho},J}(u;[s-b,s+b]\times S^1)<\hbar$ for otherwise we are done. Then plugging into  \eqref{mean_value_ineq_1} we get
\[
\frac8{\pi b^2}E_{H^{\rho},J}(u;[s-b,s+b]\times S^1)>\epsilon-cb^2\geq\epsilon/2.
\]
The lemma thus holds with $\delta=\frac{\pi b^2\epsilon}{16}$.

\end{proof}

\section{Applications}
\label{sec:applications}
\subsection{Vanishing and the Hofer-Zehnder capacity}
In \cite{Groman2015} the first named author associates a non-Archimedean valued algebra $SH^*(M|K)$, the localized symplectic cohomology, to any compact subset $K\subset M$ in a geometrically bounded symplectic manifold $M$. This is the Floer homology of the function
\[
H_K(x):=\begin{cases}0,&\quad x\in K,\\
          \infty,&\quad x\not\in K.
    \end{cases}
\]
It generalizes a construction by Cieliebak-Floer-Hofer \cite{CieliebakFloerHofer1995}, and should be compared similar ideas of Venkatesh \cite{Venkatesh2018} and Varolgunes \cite{Varolgunes2018}. It can be computed as follows. Let $H_i$ be a monotone sequence of dissipative Hamiltonians converging to $0$ on $K$ and to $\infty$ outside of $K$. Then
\[
SH^*(M|K)=\varprojlim_a\varinjlim_iHF^*_{[a,\infty)}(H_i).
\]
The algebra $SH^*(M|K)$ with its valuation are independent of the choice of the sequence $H_i$. To simplify the discussion, we assume the sequence $H_i$ is constructed as follows. Let $F_K:M\to\R$ be a dissipative exhaustion function such that for some $c>0$ $F_K^{-1}([-c,0])=K$, $F_K^{-1}(0)=\partial K$, and $F_K$ has no non-trivial periodic orbits of period $\leq 1$. Such a function exists as shown in \cite{Groman2015}. Namely, one takes an appropriate smoothing of the distance function to $K$ and multiplies it by a small enough constant.   Let $h_i:\R_+\to(-\epsilon,\infty)$ be a monotone sequence converging to $0$ on $[0,1]$, to $\infty$ everywhere else and such that $h_i'=1$ on $(2,\infty)$. Then we take $H_i:=h_i\circ F_K$.

Given any smooth dissipative Hamiltonian $H$, there is a natural map
\[
HF^*(H)\to SH^*(M|K),
\]
defined by taking the direct limit over continuation maps $H\to H_i$. Note that this map is not necessarily valuation preserving, unless we assume in addition $H\leq H_K$. Given a family of dissipative Hamiltonians $\mathcal{H}$ as in the discussion preceding Theorem \ref{tmProduct} , we get an induced map
\[
SH^*(M;\mathcal{H})\to SH^*(M|K).
\]
This is a unital ring map \cite{Groman2015}. In particular,
\begin{lem}
Let $\mathcal{H}$ be as above, and suppose $SH^*(M;\mathcal{H})=\{0\}$. Then $SH^*(M|K)=\{0\}$ for any compact $K\subset M$.
\end{lem}
\begin{proof}
This is an immediate consequence of unitality.
\end{proof}

Recall the definition of the $\pi_1$-sensitive Hofer Zehnder capacity $c^o_{HZ}(K)$ of compact sets $K\subset M$ which is defined as follows. Call a smooth function $H$ supported on $K$ \textit{Hofer-Zehnder admissible} if  $H\leq 0$, $H^{-1}(\min H)$ and $H^{-1}(0)$ both contain non-empty open sets, and the Hamiltonian flow of $H$ has no non-trivial contractible periodic orbits of period $\leq 1$. We say that $c^o_{HZ}< c$ if any Hofer-Zehnder admissible $H$ has oscillation $<c$.

The rest of this section is devoted to the following Theorem.
\begin{thm}\label{tmHZVan}
Suppose $SH^*(M|K)=\{0\}$. Then $c_{HZ}^o(K)<\infty.$
\end{thm}
To prove the theorem we first need to discuss spectral invariants for dissipative Hamiltonians. In \cite{Groman2015} it was shown that there is a natural map $f:H^*(M;\Lambda_{\mathrm{univ}})\to HF^*(H)$ for any dissipative Hamiltonian $H$. Note that unlike in the compact case this map may not be an isomorphism. Nevertheless, this allows us to define the spectral invariant
\[
\rho(H;1):=\val(f(1_{H^*(M;\Lambda_{\mathrm{univ}})})).
\]
Unlike in the compact case, we may well have $\rho(H;1)=-\infty$. The spectral invariant is monotone in the sense that
\[
H_1\leq H_2 \Rightarrow \rho(H_2;1)\leq \rho(H_1;1).
\]
Further, we have the following property for whose proof we refer to Usher's \cite[Theorem 4.1]{Usher2010a}.
\begin{lem}\label{lmSlowHam}
If $H$ is an autonomous dissipative Hamiltonian, otherwise possibly degenerate, and if $H$ possesses no non-trivial contractible periodic orbits of period $\leq 1$ then $\rho(H;1)=-\min H.$
\end{lem}

\begin{lem}
Suppose $SH^*(M|K)=\{0\}$. Then for any $a>0$ there is an $i_0$ such that for any $i\geq i_0$ we have, for $H_i=h_i\circ F_K$ as above,
\[
\rho(H_i;1)\leq -a.
\]

\end{lem}
\begin{proof}
Under the vanishing assumption we have that $1$ maps to $0$ in $\varinjlim_i HF^*_{[-a,\infty)}(H_i)$ for any $a>0$. This in turn means that for $i$ large enough $1$ maps to $0$ in $HF^*_{[-a,\infty)}(H_i)$. By definition, this means $\rho(H_i;1)\leq -a$.
\end{proof}
\begin{proof}[Proof  of Theorem \ref{tmHZVan}]
Let $H$ be a Hofer-Zehnder admissible function defined on $K$. We smoothly extend $H$ to a proper exhaustion function on $M$ of the form $h\circ F_K$ with $h$ convex and such that $h'=1$  outside of on open set containing $K$. We fix this $h$ independently of the function $H$. By definition of $F_K$, $H$ thus extended has no non-trivial periodic orbits outside of $K$. Fix $i_0$ such that $\rho(H_i;1)\leq-1$ for all $i\geq i_0$.  There is a constant $A_0$ such that if $A:=\operatorname{Osc}_K(H)>A_0$ we have that $H+A\geq H_i$. Thus,
\[
\rho(H+A;1)\leq \rho(H_i)\leq-1<0=-\min(H+A).
\]
But by Lemma \ref{lmSlowHam}, any such function has a non-trivial contractible periodic orbit of period $\leq 1$. It follows that $\operatorname{Osc}_K(H)\leq A_0$.
\end{proof}
\subsection{Obstructions to Lagrangian embeddings}
\begin{thm}\label{thm:LagBdDisc}
Suppose $SH^*(T^*M \! \colon \!  \omega_\sigma)=\{0\}$. Then for any closed embedded $\omega_\sigma$-Lagrangian $L\subset M$ there is a real number $E$ such that for any $\omega_\sigma$-tame roughly conical almost complex structure $J$, we have that $L$ either bounds a non-constant $J$-holomorphic disk or meets a non-constant $J$-holomorphic sphere of energy $\leq E$.
\end{thm}
\begin{proof}
Suppose otherwise. That is, there is a Lagrangian submanifold $L$ so that for each $E$ there is a $J_E$ so that $L$ bounds no $J_E$ holomorphic disc and intersects no $J_E$-holomorphic sphere.   Then standard Floer theory gives rise to  the following. Associated to $L$ there is a unital $\Lambda_{\operatorname{univ}}$-algebra $HF^*(L,L\! \colon \! \omega_\sigma)$, the Lagrangian intersection Floer homology, which carries a unital map $\mathcal{CO}:SH^*(M \! \colon \!  \omega_\sigma)\to HF^*(L,L\! \colon \! \omega_\sigma)$. 

We elucidate this. Pick a sequence of energies $E_n\to\infty$ and a sequence of almost complex structures $J_n$ such that $L$ bounds no $J_n$-holomorphic discs of energy $\leq E_n$. Such a sequence exists by the current assumption and Gromov compactness. For a Hamiltonian $H$ which is linear at infinity we define the Floer complex $CF^*(L,L,H\! \colon \!  \omega_\sigma)$ generated by Hamiltonian chords starting and ending on $L$.  We define the Floer differential by counting solutions to Floer's equation strips with boundaries on $L$ with. The compactness for these is the same as for Floer cylinders. The assumption on discs implies that after possibly a slight perturbation of $J_n$ to achieve regularity the count of Floer strips with action difference in $[0,E_{n}]$ defines a differential on over $\Lambda_{\operatorname{univ}}\big/t^{E_n}\Lambda_{\operatorname{univ}}$. Taking an inverse limit over $E_n$, we obtain the module $HF^*(L,L,H\! \colon \!  \omega_\sigma)$ over the Novikov ring $\Lambda_{\operatorname{univ}}$.  

We now consider a monotone sequence $H_n$ of Hamiltonians which are linear at infinity, then there are well defined continuation maps $HF^*(L,L,H_n\! \colon \!  \omega_\sigma)\to HF^*(L,L,H_{n+1}\! \colon \!  \omega_\sigma)$ which can in fact shown to be isomorphisms since $L$ is assumed to be closed. This is established in \cite[Theorem 4.14]{Groman2018}. Moreover, the module $HF^*(L,L\! \colon \!  \omega_\sigma)=\lim_nHF^*(L,L,H_n\! \colon \!  \omega_\sigma)$ is in fact an algebra with the product defined in terms of solutions to Floer's equation on triangles. On the other hand, for a pair of Hamiltonians $H\leq H'$, there is a well defined map $\mathcal{CO}:HF^*(T^*N,H\! \colon \!  \omega_\sigma)\to HF^*(L,L,H'\! \colon \!  \omega_\sigma)$ given by counting appropriate discs with boundary on $L$ with one interior and one boundary puncture. These maps commute with continuation maps. Thus we get an induced map as claimed. It is standard to show that this is a unital algebra map.

A particular consequence of unitality of the map $\mathcal{CO}$ and the vanishing of $SH$ is that $HF^*(L,L\! \colon \!  \omega_\sigma)=0$. On the other hand, there is a PSS-homomorphism of $\Lambda_{\operatorname{univ}}$-modules $H^*(L;\Lambda_{\operatorname{univ}})\to HF^*(L,L\! \colon \!  \omega_\sigma)$.  Under the assumption of no discs, the PSS map is an isomorphism. In particular that $HF^*(L,L\! \colon \!  \omega_\sigma)\neq 0$. This contradiction implies the claim.
\end{proof}

If one is willing to take on board the machinery of \cite{FukayaOhOhtaOno2009} in our setting, we can make the following statement.
\begin{thm}
\label{thm:lagapp}
Suppose $L\subset T^*N$ is a closed Lagrangian submanifold such that the induced map $i_*:H_*(L;\Z)\to H_*(T^*N;\Z)$ is injective. Then $SH^*(T^*N \! \colon \!  \omega_\sigma)\neq0$.
\end{thm}
\begin{proof}
Under the assumption it is proven in \cite{FukayaOhOhtaOno2009} that after possibly deforming by a bounding co-chain $b$, there is a well defined and non-vanishing Lagrangian Floer homology $HF^*_b(L,L\! \colon \!  \omega_\sigma)$ which is again a unital $\Lambda_\sigma$-algebra. Moreover there is a $b$-deformed unital map $\mathcal{CO}_b:SH^*(T^*N \! \colon \! \omega_\sigma)\to HF^*_b(L,L\! \colon \! \omega_\sigma)$. The claim now follows as in the previous theorem.
\end{proof}

\subsection{Non-contractible periodic orbits}

We shall need to slightly relax the conditions in Definition \ref{dfDess} when dealing periodic orbits representing a non-trivial class $a \in [S^1,T^*N]$.
\begin{enumerate}
\item There is a constant $\delta>0$ and an exhaustion of $T^*N$ by compacts sets $K_i$ such that $d(\partial K_{i+1},K_i)>2\delta$ and $J_H$ is uniformly isoperimetric on the $\delta$ neighbourhood of $\bigcup_{i=1}^\infty\partial K_i$.
\item Condition \ref{dfDess2} of Definition \ref{dfDess} holds for loops $\gamma$ representing $a$.
\end{enumerate}
Note that with this weakened condition, the proof of Theorem \ref{tmDissDaimEst} goes through with only slight changes. Namely, it suffices to apply monotonicity near the $\partial K_i$.

In the following let $\mathcal{H}^*$ be  the family of dissipative Hamiltonians $H$ for which there are constants $c_1,c_2$ with $c_1|p|<H(p,q)<c_2|p|$. Note that $\mathcal{H}^*$ is independent of the choice of metric.
\begin{thm}
\label{thm:non-contract-proof}
Let $0 \ne a \in [S^1,N]$. Fix a magnetic form $\sigma$. Then for any function $H\in\mathcal{H}^*$ we have that there is a sequence $b_i\to\infty$ for which $H^{-1}(b_i)$ contains an $\omega_\sigma$-periodic orbit representing the class $a$.
\end{thm}
\begin{proof}
We start by observing that by Viterbo's theorem, for the ordinary geodesic flow we have\footnote{If the transgression class $\tau(w_2(N))$ does not vanish, we take our field $\mathbb{K}$ to have characteristic 2.} $SH^{*,a}(T^*N;\mathcal{H}^* \colon \! \omega_0)\neq 0$. By Theorem \ref{tmMainIso} we can compute $SH^{*,a}(T^*N;\mathcal{H}^* \colon \! \omega_0)$ using a Floer complex defined via the magnetic dynamics, albeit with certain twisted coefficients. We refer hence to the $\omega_\sigma$ flow for any Hamiltonian mentioned in this proof.
Suppose the conclusion of the Theorem does not hold. Then there is a $t_0>0$ for which $H(t)$ contains no $\omega_\sigma$-periodic orbit representing $a$ whenever $t>t_0$. We can then compute $SH^{*,a}(T^*N;\mathcal{H}^* \colon \! \omega_0)$ using a sequence of Hamiltonians of the form $h_i\circ H$ which are constructed as follows. Assume first that $iH$ is dissipative for all $i$. Then we take $h_i$ so that $h_i'|_{[0,t_0]}$ is so small that $h_i\circ H$ has no periodic orbits representing $a$, and such that $h'_i=i$ near infinity. This sequence is cofinal in $\mathcal{H}^*$. Moreover it is non-degenerate for $a$. Since it has no periodic orbits representing $a$, the twisted symplectic cohomology must vanish.

We now remove the assumption on dissipativity of $iH$. Instead we take $h'_i=i$ on every even interval contained in $\R_+\setminus [0,t_0]$ and $h'_i=0$ on every odd interval. We perturb slightly to smooth. The corresponding $H_i$ is automatically dissipative when paired with a uniformly isoperimetric $J$ since on the odd intervals we get the ordinary $J$-holomorphic equation. Moreover, since we are considering a non-trivial $a$, these Hamiltonians remain non-degenerate. Thus the statement holds for \textit{any} $H\in\mathcal{H}^*$.
\end{proof}
\begin{thm}
With the assumptions and notation of the previous theorem, let $A\subset\R$ be an open set such that for any $t\in A$ we have that $H^{-1}(t)$ does not contain a periodic orbit representing $a$. Then $\mu(A\cap[0,n])/n\to 0$ as $n\to\infty$. Here $\mu$ is the Lebesgue measure.
\end{thm}
\begin{proof}
Suppose by contradiction that $\mu(A)\cap[0,n]\geq \epsilon n$ for every $n$. Let $h(t):=\int_0^t1_A$. Then $\epsilon t\leq h(t)\leq t$ for all $t\in\R_+$. Thus, the sequence of functions $h_i=ih$ is $\preceq$-cofinal in the set of all linear functions. On the other hand we may assume without loss of generality that for any $n$ the set $[n,n+1]\setminus A$ contains an interval of length $1/2$. Indeed, by erasing from $A$ the half of lesser measure from each interval $[n,n+1]$, we maintain the inequality $\mu(A)\cap[0,n]\geq \frac{\epsilon}{2} n$. Thus the function $h_i$ is constant an infinite sequence of intervals of length $1/2.$ The functions $H_i:=h_i\circ H$ are thus dissipative as in proof of the previous theorem. The sequence $H_i$ is ${\preceq}$-cofinal in $\mathcal{H}^*$ and none of the Hamiltonians in the sequence have a periodic orbit representing $a$, even after a slight perturbation to make them smooth. In particular, the twisted symplectic cohomology must vanish. This contradiction implies the claim.
\end{proof}

\bibliographystyle{plain}
\bibliography{willmacbibtex}

\end{document}